\documentclass[11pt,a4paper]{article}
\usepackage[left=2.54cm, right=2.54cm, top=2.54cm, bottom=2.38cm]{geometry} 

\usepackage{comment}
\usepackage{amssymb}
\usepackage{amsmath}
\usepackage{amsthm} 
\usepackage{nccmath} 
\usepackage{mathtools}
\usepackage[noadjust]{cite}

\usepackage[scr=rsfs]{mathalpha} 

\DeclareMathAlphabet\mathbfcal{OMS}{cmsy}{b}{n} 

\usepackage{color, colortbl}

\definecolor{LightCyan}{rgb}{0.88,1,1}
\definecolor{LightGreen}{rgb}{0.56,0.93,0.56}
\definecolor{LightOrange}{rgb}{0.99,0.84,0.64}
\definecolor{LightBlue}{rgb}{0.764, 0.831, 0.996}

\usepackage{graphicx}
\usepackage{subcaption}	

\usepackage{enumerate}   
\usepackage{enumitem}

\newcommand{\Hcal}{\mathcal{H}}

\newcommand{\Nbb}{\mathbb{N}}
\newcommand\inner[2]{\langle #1, #2 \rangle} 

\newtheorem{theorem}{Theorem}
\newtheorem{proposition}[theorem]{Proposition}
\newtheorem{lemma}[theorem]{Lemma}
\newtheorem{remark}[theorem]{Remark}
\newtheorem{corollary}[theorem]{Corollary}

\newtheorem{definition}[theorem]{Definition}

\title{Tseng's Algorithm with Extrapolation from the Past Endowed with Variable Metrics and Error Terms}

\author{Buris Tongnoi \\
Email: buris.tongnoi@univie.ac.at}
\date{Faculty of Mathematics, University of Vienna, Oskar-Morgenstern-Platz 1,\\
	Vienna 1090, Austria.}

\begin{document}
\maketitle

\textbf{Abstract:} In this paper, we propose a variable metric version of Tseng's algorithm (the forward-backward-forward algorithm: FBF) combined with extrapolation from the past that includes error terms for finding a zero of the sum of a maximally monotone operator and a monotone Lipschitzian operator in Hilbert spaces. This can be seen as the optimistic gradient descent ascent (OGDA) algorithm endowed with variable metrics and error terms. Primal-dual algorithms are also proposed for monotone inclusion problems involving compositions with linear operators. The primal-dual problem occurring in image deblurring demonstrates an application of our theoretical results.

\section{Introduction}
Various problems in real-world applications like signal and image processing [\cite{Bot2015Csetnek}], Positorn Emission Tomography [\cite{Aharon.et.al2001}] and machine learning [\cite{Shai et al.2012}] can be expressed as non-smooth optimization problems and these problems can also be modeled as inclusion problems involving monotone set-valued operators in Hilbert space $\mathcal{H}$ say
\begin{align}
\mbox{find}\quad \bar{x}\in\mathcal{H} \quad\mbox{such that}\quad z\in Fx
\end{align}
where $F:\mathcal{H}\rightarrow 2^{\mathcal{H}}$ is monotone and $z\in\mathcal{H}$; see, e.g. [\cite{AriasCombettes2011,Malitsky2020Tam,Combettes2014Vu,Vu2013,Vu2013Variable}]. In many situations, the operators $F$ can be represented as the sum of two monotone operators, one of which is the composition of a monotone operator with a  linear transformation and its adjoint operator. In such circumstances, it is usually desirable to also solve the associated dual inclusion [\cite{AriasCombettes2011, Combettes2012Pesquet, Vu2013, Bot2015Csetnek}]. Let $A:\mathcal{H}\rightarrow 2^{\mathcal{H}}$. Then, $A$ is \textit{monotone} if
$(\forall (x,u),(y,v)\in Gra A)\; \inner{x-y}{u-v}\geq 0,$
where $Gra A=\{(x,\xi) \;|\; \xi\in A(x) \}$ is the graph of $A$. The monotone operator $A$ is \textit{maximally
monotone (or maximal monotone)} if there exists no monotone operator $B:\mathcal{H}\rightarrow 2^{\mathcal{H}}$ such that $Gra B$ properly contains $GraA$, i.e., for every $(x, u)\in \mathcal{H}\times\mathcal{H}$,
$(x, u) \in GraA \Leftrightarrow (\forall (y, v) \in GraA )\; \inner{x-y}{u-v} \geq 0$. Whenever the operator $A$ satisfies the inequality : $\|Ax-Ay\|\leq\upsilon\|x-y \|, \; \forall x,y\in\mathcal{H}$ for some $\upsilon>0$, it call \textit{$\upsilon$-Lipschitzian} and we also know that if $f:\mathcal{H}\rightarrow (-\infty,\infty]$ belong to the set of proper lower semicontinuous convex functions on $\mathcal{H}$ denoted by $\Gamma_0(\mathcal{H})$, then $\partial f$ is maximally monotone (see [\cite{BC-Book}] Theorem 20.40). The basic (finite sum) problem that we consider in this paper is the following.

\textbf{Problem 1} Let $\mathcal{H}$ be a real Hilbert space, let $m$ be a strictly positive integer, let $z\in\mathcal{H}$, let $A: \mathcal{H}\rightarrow 2^{\mathcal{H}}$ be a maximally monotone operator, let $C:\mathcal{H}\rightarrow\mathcal{H}$ be monotone and $v_0$-Lipschitzian  for some $v_0\in(0,+\infty)$. For every $i\in\{1,\dots,m \}$, let $\mathcal{G}_i$ be a real Hilbert space, let $r_i\in\mathcal{G}_i$, let $B_i:\mathcal{G}_i\rightarrow 2^{\mathcal{G}}$ be a maximally monotone operator, let $L_i :\mathcal{H}\rightarrow\mathcal{G}_i$ be a nonzero bounded linear operator. Suppose that
\begin{align}\label{Mono_Inclusion 1}
z \in ran \left( A + \sum\limits_{i=1}^m L^*_{i} \left( B_i (L_i \cdot - r_i) \right) + C \right)
\end{align}
The problem is to solve the primal inclusion
\begin{align}\label{Mono Inclusion 2}
\mbox{find}\quad \bar{x}\in\mathcal{H}\; \mbox{such that}\; z\in A\bar{x} +\sum\limits_{i=1}^m L_i^* \left( B_i (L_i \bar{x} - r_i) \right) + C\bar{x}
\end{align}
and the dual inclusion
\begin{align}\label{Mono Inclusion 3}
\mbox{find}\quad \bar{v}_1\in \mathcal{G}_1, \dots,\bar{v}_m\in\mathcal{G}_m\; \mbox{such that}\; (\exists x \in \mathcal{H})
\begin{cases}
z - \sum\limits_{i=1}^m L_i^* \bar{v}_i \in Ax + Cx\\
(\exists i \in \{ 1,\dots,m\})\; \bar{v}_i \in B_i(L_ix-r_i)
\end{cases}
\end{align}

By using  properties for any function belongs to $\Gamma_0(\mathcal{H})$ (see Proposition 15.2 and Corollary 16.24 in [\cite{BC-Book}]) and some qualification conditions (for assuring subdifferential calculus), we can show that Problem 1 and the convex minimization problems below are equivalent by letting $A=\partial f$, $B=\partial g_i\;\forall\; i=1,\dots,m$, $C=\nabla h$  where $h$ is a differentiable convex function with Lipschitz continuous gradient. The convex minimization problem is the following:

\textbf{Problem 2} Let $\mathcal{H}$ be a real Hilbert space, let $z\in\mathcal{H}$, let $m$ be a strictly positive integer, let $f\in\Gamma_0(\mathcal{H})$, and let $h: \mathcal{H}\rightarrow\mathbb{R}$ be convex and differentiable with a $v_0$-Lipschitzian gradient for some $v_0\in (0,+\infty)$. For every $i\in \{ 1,\dots,m\}$, let $\mathcal{G}_i$ be a real Hilbert space, let $r_i\in\mathcal{G}_i$ let $g_i\in \Gamma_0(\mathcal{G}_i)$ and suppose that $L_i:\mathcal{H}\rightarrow\mathcal{G}_i$ is a nonzero bounded linear operator. Consider the problem
\begin{align}\label{Primal-Dual 1}
\underset{x\in\mathcal{H}}{minimize}\; f(x)+\sum\limits_{i=1}^{m}g_i(L_ix) + h(x)
,
\end{align}
and the Fenchel-Rockafellar dual problem [\cite{Rockafellar}]:
\begin{align}\label{Primal-Dual 2}
\underset{v_i\in\mathcal{G}_i(\forall i=1,\dots,m)}{minimize}\; (f^*\Box h^*) \left(-\sum\limits_{i=1}^{m} L_i^* v_i\right)+\sum\limits_{i=1}^m  g^*_i(v_i).
\end{align}

The aforementioned problems are so-called \textit{primal-dual problems}. Using the product space approach, primal-dual inclusion problems (\ref{Mono Inclusion 2}) and (\ref{Mono Inclusion 3}) can be written as the finding $\bar{x}\in\mathcal{H}$ with $0\in\mathbf{A}(\bar{x}) +\mathbf{B}(\bar{x})$, where $\mathbf{A}$ is maximally monotone and $\mathbf{B}$ is either cocoercive or monotone and Lipschitz continuous. When $C$ is \textit{cocoercive} (i.e., $\inner{Cx-Cy}{x-y} \geq \beta \|Cx -Cy\|^2\; \forall x, y \in H$ and $\beta >0$), then $\mathbf{B}$ is cocoercive (in a renormed product Hilbert space), which is proposed in Vu's work [\cite{Vu2013}]. His method stems from the \textit{forward-backward (FB) splitting algorithm} 
\begin{align}\label{FB-Algor}
x_{n+1}=(1-\lambda_n)x_n+\lambda_n J_{\gamma A}(x_n-\gamma B x_n)\quad
\forall n\geq 0,
\end{align}
where the \textit{resolvent operator}  $J_A=(Id+A)^{-1} $ 
is nonexpansive, single-valued and the set of fixed points of $J_{A}$ coincides with the set of zeros of $A$. In the case of $A=\partial f$, then $J_{\partial f}(x)=prox_{f}(x) = \underset{y\in\mathcal{H}}{argmin} \{ f(y) + \frac{1}{2}\|y-x \|^2 \},\; \forall x\in\mathcal{H}$ is the \textit{proximal operator}. Meanwhile, in the work [\cite{AriasCombettes2011}] of  Brice$\tilde{\text{n}}$o-Arias and Combettes, $\mathbf{B}$ is monotone and Lipschitzian. Their scheme is based on the \textit{Tseng's algorithm or forward-backward-forward (FBF) algorithm}. It can be expressed in the simple formula as below
\begin{align}\label{Tseng_Algor}
y_n &= J_{\gamma A}(x_n-\gamma Bx_n) \notag\\
x_{n+1} &= y_n+\gamma(Bx_n-By_n).
\end{align}

We note that every cocoercive operator is monotone Lipschitzian, but the converse is not true in general (see [\cite{BC-Book}]). In our work, we investigate Tseng's method and try to improve this algorithm into better ones in the context of its efficiency and generalization.

From Tseng's algorithm in [\cite{Tseng2000}], we can see that the algorithm must compute twice of $B(x_n)$ and $B(y_n)$, which wastes the algorithm process. To alter this issue, Popov [\cite{Popov}] proposed a technique in the extragradient method that only requires a single gradient computation per update. Then we intend to combine this technique with Tseng's algorithm and call it Tseng's algorithm with extrapolation from the past. We obtain a general scheme as (see [\cite{Bohmetal2020}])
\begin{align}
\mbox{Tseng-General}\;
\begin{cases} 
y_n = J_{\gamma A}(x_n-\gamma B(z_n))\\
x_{n+1} =y_n+\gamma(B(z_n)-B(y_n)).
\end{cases}
\end{align}
\begin{enumerate}
\item For $z_n=x_n$ we obtain Tseng's algorithm (\ref{Tseng_Algor}), see [\cite{Tseng2000}]
\item For $z_n=y_{n-1}$ we obtain Tseng's algorithm with extrapolation. This algorithm is nothing else than the scheme Malitsky-Tam [\cite{Malitsky2020Tam}], also known as Optimistic Gradient Descent Ascent (OGDA) metho
d for saddle point problems, with applications in machine learning.
\end{enumerate}

We are interested in developing Tseng's algorithm with extrapolation from the past endowed with variable metrics and error terms. The idea behind our scheme originated from the modified Tseng's method (OGDA) algorithm in [\cite{Malitsky2020Tam}] that the cocoercivity of the single-valued operator is no longer required, and each iteration needs only one forward evaluation rather than two, as is the case in Tseng's method. Moreover, when the resolvent operator cannot compute efficiently, it is allowed to have errors. For example, the classical Tseng's algorithm in [\cite{AriasCombettes2011}], the algorithm will be more flexible if we concede it has error terms. Additionally, some works proposed the use of variable metrics to get more efficient proximal algorithms (see [\cite{Variablemetic_quaisi_Combettes2013Vu, Parente2008Lotito,Burke1999Qian, Burke2000Qian}]), which can apply to the forward-backward splitting algorithm in [\cite{Combettes2014Vu}] and Tseng's algorithm in [\cite{Vu2013Variable}]. Therefore, we round up the modification algorithm's benefits and put them into our scheme shown in the main theorem of this paper.

In this article, we propose the variable metric Tseng's algorithm with extrapolation from the past and error terms shown in section 3. We give some notations and background knowledge on convex analysis and monotone operator theory in section 2. Next, we use our main result to develop a variable metric primal-dual algorithm for solving the type of composite inclusions for Problem 1 and Problem 2, respectively. Moreover, we illustrate the application of our algorithm in image deblurring in section 6.

\section{Preliminaries}

In this section, we will give some background knowledge and tools which are useful for the main results in the section 3. 

Throughout this paper, $\mathcal{H}$, $\mathcal{G}$, $(\mathcal{G}_i)_{1\leq i\leq m}$ are real Hilbert spaces, and $\mathbb{R}$, $\mathbb{N}$ represent a set of real number and a set of natural number, respectively. The scalar product and associated norms are respectively denoted by $\inner{\cdot}{\cdot}$ and $\|\cdot\|$.  Let $\mathcal{G}_1\bigoplus \dots\bigoplus\mathcal{G}_m$ be the direct sum of the Hilbert spaces $(\mathcal{G}_i)_{1\leq i\leq m}$. For every $i\in \{ 1,\dots, m\}$, let $T_i$ be a mapping from $\mathcal{G}_i$ to some set $\mathcal{R}$. Then
\begin{align} \label{Notation 1}
\bigoplus_{i=1}^m T_i : \bigoplus_{i=1}^m \mathcal{G}_i \rightarrow \mathcal{R}: (y_i)_{1\leq i \leq m} \mapsto \sum\limits_{i=1}^{m} T_iy_i. 
\end{align}
We denote the space of bounded linear operators from $\mathcal{H}$ to $\mathcal{G}$ by $\mathcal{B}(\mathcal{H},\mathcal{G})$, the adjoint of $L\in \mathcal{B}(\mathcal{H},\mathcal{G})$ is denoted by $L^*$. We set $\mathcal{B}(\mathcal{H})= \mathcal{B}(\mathcal{H},\mathcal{H})$. The symbols $\rightharpoonup$ and $\rightarrow$ denote, respectively, weak and strong convergence, and $Id$ denotes the identity operator. 
 We set $\mathcal{S}(\mathcal{H})=\{L\in \mathcal{B}(\mathcal{H}) | L=L^* \}$.  The Loewner partial ordering on $\mathcal{S}(\mathcal{H})$ is denoted by
\begin{align}
(\forall U\in \mathcal{S}(\mathcal{H})) (\forall V\in\mathcal{S}(\mathcal{H}))\quad U \succcurlyeq V \Leftrightarrow (\forall x\in\mathcal{H})\quad  \inner{Ux}{x} \geq \inner{Vx}{x}.
\end{align}
Now let $\alpha \in [0,+\infty)$. We set
\begin{align}
\mathcal{P}_{\alpha}(\mathcal{H})= \{ U\in \mathcal{S}(\mathcal{H}) | U \succcurlyeq \alpha Id \},
\end{align}
and we denote by $\sqrt{U}$ the square root of $U\in\mathcal{P}_{\alpha}(\mathcal{H})$. Moreover, for every $U\in\mathcal{P}_{\alpha}(\mathcal{H})$, we define a semi-scalar product and a semi-norm (a scalar product and a norm if $\alpha>0$) by 
\begin{align}
(\forall x\in\mathcal{H}) (\forall y\in\mathcal{H})\quad \inner{x}{y}_{U} = \inner{Ux}{y}\; \mbox{and}\; \|x\|_{U} = \sqrt{\inner{Ux}{x}}.
\end{align}

Let $A:\mathcal{H}\rightarrow 2^{\mathcal{H}}$ be a set-valued operator. The domain of $A$ is $dom A = \{ x\in \mathcal{H} \;|\; Ax \neq \emptyset \}$. The \textit{inverse} of $A$, denoted by $A^{-1}$, is defined through its graph such that $ Gra A^{-1} =\{ (u,x)\in \mathcal{H}\times \mathcal{H} \;|\; (x,u)\in Gra A  \}$. The set of zeros of $A$ is $zer A = \{ x\in\mathcal{H} \;|\; 0\in Ax\}$, and the range of $A$ is $ran A=\{ u\in \mathcal{H} \;|\; (\exists x\in\mathcal{H})\; u\in Ax \}$, and the resolvent of $A$ is 
\begin{align}
J_A = (Id+A)^{-1}.
\end{align}
Moreover, $A$ is \textit{monotone} if 
\begin{align}
(\forall (x,y) \in \mathcal{H}\times \mathcal{H}) (\forall (u,v) \in Ax \times Ay) \quad \inner{x-y}{u-v} \geq 0,
\end{align}
and \textit{maximally monotone} if it is monotone and there exists no monotone operator $B:\mathcal{H}\rightarrow 2^{\mathcal{H}}$ such that $Gra A \subset Gra B$ and $A\neq B$.
The \textit{conjugate} of $f:\mathcal{H} \rightarrow[-\infty,\infty]$ is 
\begin{align}
f^*:\mathcal{H}\rightarrow[-\infty,+\infty] : u \mapsto \sup\limits_{x\in\mathcal{H}} \left( \inner{x}{u} - f(x) \right),
\end{align}
and the \textit{infimal convolution} of $f$, $g: \mathcal{H}\rightarrow (-\infty, +\infty]$ is 
\begin{align}\label{Infimal Convolution}
f\Box g :\mathcal{H} \rightarrow[-\infty,+\infty] : x \mapsto \inf\limits_{y\in\mathcal{H}} (f(y)+g(x-y)).
\end{align}
The class of lower semicontinuous convex functions $f: \mathcal{H}\rightarrow (-\infty,+\infty]$ such that $domf=\{ x\in \mathcal{H} \;|\; f(x)< + \infty \} \neq \emptyset$ is denoted by $\Gamma_0(\mathcal{H})$. If $f\in\Gamma_{0}(\mathcal{H})$, then $f^*\in \Gamma_0(\mathcal{H})$, and the \textit{subdifferential} of $f$ is the maximally monotone operator, which define as
\begin{align}
\partial f : \mathcal{H} \rightarrow 2^{\mathcal{H}} : x \mapsto \{ u\in \mathcal{H} \;|\; (\forall y\in \mathcal{H})\; \inner{y-x}{u} + f(x) \leq f(y) \},
\end{align}
with inverse
\begin{align}\label{(subdriff(f))inverse equals subdrif f_star}
(\partial f)^{-1} = \partial f^*.
\end{align}
The \textit{indicator function} and the \textit{distance function} of $C$ are defined on $\mathcal{H}$ as 
\begin{align}
\iota_C : x\mapsto\begin{cases} 0,  &if x\in C;\\ +\infty, &if x\not\in C \end{cases} \quad \mbox{and} \quad d_C = \iota_C\Box \|\cdot\| : x\mapsto \inf\limits_{y\in C} \|x-y\|,
\end{align}
respectively. The support function of $C$, $\sigma_C : \mathcal{H}\rightarrow [-\infty, \infty] : u \mapsto \sup\inner{C}{u}$, equals to $\iota_{C}^{*}$.

The \textit{proximity operator} of $f\in\Gamma_0(\mathcal{H})$ relative to the metric induced by $U\in \mathcal{P}_{\alpha}(\mathcal{H})$ is [\cite{Hiriart-Urruty1993Lemarechal}, Section XV.4]
\begin{align}\label{prox_variable_matrices}
prox^{U}_{f} : \mathcal{H} \rightarrow \mathcal{H} : x \mapsto \underset{y\in\mathcal{H}}{\arg\min} f(y) + \frac{1}{2}\|x-y\|^2_{U},
\end{align}
and the projector onto a nonempty closed convex subset $C$ of $\mathcal{H}$ relative to the norm $\|\cdot\|_U$ is denoted by $P_C^U$. We have
\begin{align}\label{prox^U_f and P^U_C}
prox_f^U = J_{U^{-1}\partial f} \;\mbox{ and }\; P_C^U=prox^U_{\iota_C},
\end{align}
and we write $prox_f^{Id}=prox_f$.
Finally, $\ell_+$ denotes the set of all sequences in
$[0,+\infty)$ and $\ell^1$ (resp. $\ell^2$) the space of all absolutely (resp. square) summable sequences in $\mathbb{R}$. Therefore $\ell^1_{+}$ means the space of all absolutely summable sequences in $[0,\infty)$.

\begin{definition}[\cite{Variablemetic_quaisi_Combettes2013Vu}]\label{gen. def of quasi-Fejer}
Let $\alpha\in (0,+\infty)$, let $\phi:[0,+\infty)\rightarrow [0,+\infty)$ , let $(W_n)_{n\in\mathbb{N}}$ be a sequence in $\mathcal{P}_{\alpha}(\mathcal{H})$, let $C$ be a nonempty subset of  $\mathcal{H}$, and let $(x_n)_{n\in\mathbb{N}}$ be a sequence in $\mathcal{H}$. Then $(x_n)_{n\in\mathbb{N}}$ is $\phi$-quasi-Fejer monotone with respect to the target set $C$ relative to $(W_n)_{n\in\mathbb{N}}$ if $(\exists (\eta_n)_{n\in\mathbb{N}}\in\ell_{+}^1(\mathbb{N}))(\forall z\in C)(\exists(\epsilon_n)_{n\in\mathbb{N}}\in\ell_{+}^1(\mathbb{N})) (\forall n \in \mathbb{N})$,
\begin{align}\label{eq of quasi-Fejer}
  \phi(\|x_{n+1}-x\|_{W_{n+1}}) \leq (1+\eta_n) \phi(\|x_n-z\|_{W_n})+ \epsilon_n.
\end{align}
\end{definition}

\begin{lemma}[{[\cite{Kato1980}]} Section VI.2.6, {[\cite{Combettes2014Vu}]}, {[\cite{Variablemetic_quaisi_Combettes2013Vu}]}] \label{lem2.1(10)}
Let $\alpha\in(0,+\infty)$, let $\mu\in(0,+\infty)$ and let $A$ and $B$ be operators in $\mathcal{S}(\mathcal{H})$ such that $\mu Id \succcurlyeq A \succcurlyeq B \succcurlyeq \alpha Id$. Then the following hold:
\begin{enumerate}[label=(\roman*)]
\item $\alpha^{-1} Id \succcurlyeq B^{-1} \succcurlyeq A^{-1} \succcurlyeq \mu^{-1} Id$.
\item  $(\forall x\in \mathcal{H}) $ $\inner {A^{-1}x}{x} \geq \|A\|^{-1}\|x\|^2$.
\item $\|A^{-1}\|\leq\alpha^{-1}$.
\end{enumerate}
\end{lemma}

\begin{lemma}[{[\cite{Combettes2014Vu}]}]\label{lem3.7(11)}
Let $A:\mathcal{H}\rightarrow 2^{\mathcal{H}}$ be maximally monotone, let $\alpha\in(0,+\infty)$, let $U\in \mathcal{P}_{\alpha}(\mathcal{H})$, and let $\mathcal{G}$ be the real Hilbert space obtained by endowing $\mathcal{H}$ with the scalar product $(x,y)\mapsto \inner{x}{y}_{U^{-1}}=\inner{x}{U^{-1}y}$. Then the following hold.
\begin{enumerate}[label=(\roman*)]
\item $UA:\mathcal{G}\rightarrow 2^{\mathcal{G}}$ is maximally monotone.
\item $J_{UA}:\mathcal{G}\rightarrow\mathcal{G}$ is 1-cocoercive, i.e., firmly nonexpansive, hence nonexpansive.
\item $J_{UA} = (U^{-1}+A)^{-1}\circ U^{-1}$.
\end{enumerate}
\end{lemma}

\begin{lemma}[{[\cite{Variablemetic_quaisi_Combettes2013Vu}]}, {[\cite{Polyak1987}]} lemma 2 in section 2.2.1]\label{lem 2.2, Combettes&Vu}
Let $(\alpha_n)_{n\in\mathbb{N}}$ be a sequence in $[0,+\infty)$, let $(\eta_n)_{n\in\mathbb{N}}\in\ell_{+}^1(\mathbb{N})$ and let $(\epsilon_n)_{n\in\mathbb{N}}\in\ell_{+}^1(\mathbb{N})$ be such that $(\forall n\in\mathbb{N})$ $\alpha_{n+1}\leq(1+\eta_n)\alpha_n+\epsilon_n$. Then $(\alpha_n)_{n\in\mathbb{N}}$ converges.
\end{lemma}

\begin{proposition}[{[\cite{Variablemetic_quaisi_Combettes2013Vu}]}]\label{Prop 3.2 (10)}
Let $\alpha\in(0,+\infty)$, let $\phi:[0,+\infty)\rightarrow [0,+\infty)$ be strictly increasing and such that $\lim_{t\rightarrow +\infty}\phi(t) = +\infty$, let $(W_n)_{n\in\mathbb{N}}$ be in $\mathcal{P}_{\alpha}(\mathcal{H})$, let $C$ be a nonempty subset of $\mathcal{H}$, and let $(x_n)_{n\in\mathbb{N}}$ be a sequence in $\mathcal{H}$ such that (\ref{eq of quasi-Fejer}) is satisfied. Then the following hold.
\begin{enumerate}[label=(\roman*)]
\item Let $z\in C$. Then $(\|x_n-z\|_{W_n})_{n\in\mathbb{N}}$ converges.
\item $(x_n)_{n\in\mathbb{N}}$ is bounded. 
\end{enumerate} 
\end{proposition}

\begin{lemma}[{[\cite{Combettes2001}]}]\label{lem 3.1 (7)}
Let $\chi\in(0,1], (\alpha_n)_{n\geq0}\in\ell_+, (\beta_n)_{n\geq0}\in \ell_+ $ and $(\epsilon_n)_{n\geq 0}\in \ell_+^1$ be such that $$ (\forall n\in\mathbb{N})\; \alpha_{n+1}\leq\chi \alpha_n-\beta_n+\epsilon_n.$$
Then
\begin{enumerate}[label=(\roman*)]
\item $(\alpha_n)_{n\geq 0}$ is bounded.
\item $(\alpha_n)_{n\geq 0}$ converges.
\item $(\beta_n)_{n\geq 0} \in \ell^1$.
\item If $\chi\neq 1, (\alpha_n)_{n\geq 0}\in \ell^1$. 
\end{enumerate}
\end{lemma}

\begin{proposition}[{[\cite{BC-Book}]} Proposition 20.33]\label{Prop 20.33 (BC)}
Let $A:\mathcal{H}\rightarrow 2^{\mathcal{H}}$ be maximally monotone. Then the following hold:
\begin{enumerate}[label=(\roman*)]
\item $Gra A$ is sequentially closed in $\mathcal{H}^\text{strong}\times \mathcal{H}^\text{weak}$, i.e., for every sequence $(x_n,u_n)_{n\in\mathbb{N}}$ in $Gra A$ and every $(x,u)\in \mathcal{H}\times\mathcal{H}$ if $x_n\rightarrow x$ and $u_n\rightharpoonup u$, then $(x,u)\in Gra A$
\item $Gra A$ is sequentially closed in $\mathcal{H}^\text{weak}\times \mathcal{H}^\text{strong}$, i.e., for every sequence $(x_n,u_n)_{n\in\mathbb{N}}$ in $Gra A$ and every $(x,u)\in \mathcal{H}\times\mathcal{H}$ if $x_n\rightharpoonup x$ and $u_n \rightarrow u$, then $(x,u)\in Gra A$
\item $Gra A$ is closed in $ \mathcal{H}^\text{strong}\times \mathcal{H}^\text{strong}$.
\end{enumerate}
\end{proposition}

\begin{lemma}[{[\cite{Variablemetic_quaisi_Combettes2013Vu}]} Lemma 2.2]\label{lem 2.3 (10)}
Let $\alpha\in(0,+\infty)$ let $(\eta_n)_{n\in\mathbb{N}}\in\ell^1_+(\mathbb{N})$, and let $(W_n)_{n\in\mathbb{N}}$ be a sequence in $\mathcal{P}_\alpha(\mathcal{H})$ such that $\mu = \sup_{n\in\Nbb}\|W_n\|<+\infty$. Suppose that one of the following holds.
\begin{enumerate}[label=(\roman*)]
\item $(\forall n\in\mathbb{N})\quad (1+\eta_n)W_n \succcurlyeq W_{n+1}$.
\item $(\forall n\in\mathbb{N})\quad  (1+\eta_n)W_{n+1} \succcurlyeq W_{n}$.
\end{enumerate} 
Then there exists $W\in\mathcal{P}_\alpha(\mathcal{H})$ such that $W_n\rightarrow W$ pointwise.
\end{lemma}

\begin{theorem}[{[\cite{Variablemetic_quaisi_Combettes2013Vu}]} Theorem 3.3]\label{Thm 3.3 (10)}
Let $\alpha\in(0,+\infty)$, let $\phi:[0,+\infty)\rightarrow [0,+\infty)$ be strictly increasing and such that $\lim_{t\rightarrow +\infty} \phi(t) = +\infty$, let $(W_n)_{n\in\mathbb{N}}$ and $W$ be operators on $\mathcal{P}_\alpha(\mathcal{H})$ such that $W_n\rightarrow W$ pointwise, let $C$ be nonempty subset of $\mathcal{H}$ and let $(x_n)_{n\in\mathbb{N}}$ be sequence in $\mathcal{H}$ such that 
(\ref{eq of quasi-Fejer}) is satisfied. Then $(x_n)_{n\in\mathbb{N}}$ converges weakly to a point in $C$ if and only if every weak sequential cluster point of  $ (x_n)_{n\in\mathbb{N}}$ is in $C$.

\end{theorem}

\section{A Variable Metric Tseng's Algorithm with Extrapolation from the Past and Error Terms}

Tseng's algorithm was first proposed in [\cite{Tseng2000}] to solve inclusion involving the sum of a maximally monotone operator and a monotone Lipschitzian operator. This algorithm was considered to include computational errors in [\cite{AriasCombettes2011}] and was lately modified to involve variable metric in [\cite{Vu2013Variable}]. Now we will extend it into an extrapolation scheme.

The adjustment of our algorithm started by adding extrapolation into classical Tseng's algorithm (FB) expecting that it will reduce a cost of computation and at first we called "Tseng's algorithm with extrapolation". Next, we considered this adjusted algorithm by adding the involved error terms. Lastly, the algorithm was made to become more general by working with variable metric. Then, we called "Variable Metric Tseng's algortihm with Extrapolation and Error" for the new altered algorithm. Below we will corroborate our proposed algorithm in the theorem with its proof.

\begin{theorem}\label{Thm_TengEP_variable_Error}
Let $A: \Hcal \rightarrow 2^{\Hcal}$ be maximally monotone, let $\alpha,\beta\in(0,+\infty)$, $B:\Hcal\rightarrow\Hcal$ be a monotone and $\beta$-Lipschitzian operator on $\mathcal{H}$ such that $zer(A+B)\neq \emptyset$, let $(\eta_n)_{n\in\Nbb}\in \ell^1_{+}(\Nbb)$ and $(U_n)_{n\in\Nbb}$ be a sequence in $\mathcal{P}_{\alpha}(\Hcal)$ such that 
\begin{align}\label{ConditionA}
\mu = \sup_{n\in\Nbb}\|U_n\|<+\infty\quad  \text{and} \quad (\forall n\in N) \quad (1+\eta_n)U_{n+1} \succcurlyeq U_n.
\end{align}
Let $(\gamma_n)_{n\in\Nbb} \leq \lambda$ with $\lambda < \frac{1}{ \sqrt{10}\mu\beta}$ and $\liminf_{n\rightarrow +\infty} \gamma_n>0$. Let $(a_n)_{n\in\Nbb},(b_n)_{n\in\Nbb}$ and $(c_n)_{n\in\Nbb}$ be absolutely summable sequences in $\Hcal$. Let $x_0,  p_{-1} \in\Hcal$ and set
\begin{align}\label{set_para1}
(\forall n\in N) 
\begin{cases}
 y_n = x_n-\gamma_n U_n (B(p_{n-1})+a_n) ,\\ 
 p_n = J_{\gamma_nU_nA}(y_n) + b_n , \\
 q_n = p_n - \gamma_nU_n(B(p_n)+c_n) , \\
 x_{n+1} = x_n - y_n + q_n.
\end{cases}
\end{align}
Then the following hold for some $\bar{x} \in zer(A+B)$.
\begin{enumerate}[label=(\roman*)]
\item $\sum\limits_{n\in\mathbb{N}}\|x_n - p_n\|^2 < +\infty$ and $ \sum\limits_{n\in\mathbb{N}} \|y_n-q_n\|^2 <+\infty$,
\item $x_n\rightharpoonup \bar{x}$ and $p_n\rightharpoonup \bar{x}$.
\end{enumerate}

\end{theorem}

\begin{remark} \label{Alg to OGDA}
We given some remarks below.
\begin{enumerate}[label=(\roman*)]

\item From the proposed algorithm, if we put $a_n=c_n=0$ but $b_n$  still remains, the algorithm turn into
\begin{align*}
	(\forall n\in N) 
	\begin{cases}
		p_n = J_{\gamma_nU_nA}(x_n-\gamma_{n} U_n B(p_{n-1})) + b_n , \\
		x_{n+1} = p_n + \gamma_n U_n (B(p_{n-1})-B(p_n)),
	\end{cases}
\end{align*} then the algorithm is the adaptation of OGDA in [\cite{Malitsky2020Tam}] with variable metric and errors. In fact, the OGDA is nothing else than a particular case of our algorithm when setting $a_n=b_n=c_n=0$ and $U_n=Id$, i.e.,
\begin{align*}
	p_{n+1} = J_{\gamma_{n} A} (x_{n+1}-\gamma_{n}  B(p_{n})) =  J_{\gamma_{n} A} (p_n -2\gamma_n B(p_n)+\gamma_{n} B(p_{n-1})),
\end{align*}
in which two initial points $p_0$ and $p_1$ are required for this iterative formula.

\item Because the error terms and variable metrics that appear in this algorithm, they make our method more flexible to handle. Indeed, it can  generate a more alternative variable metric algorithm with error by using a different error model and involved iteration-dependent variable metrics.

\item In the error-free case  ($a_n=b_n=c_n=0$), we can observe that the results hold when the stepsize fulfills ${ 0 < \gamma_n < \frac{1}{2\mu\beta} }$ and $\liminf_{n\rightarrow +\infty} \gamma_n>0$. 

\end{enumerate}
\end{remark}

\begin{proof}
The structure of the proof starts with a new setting of variables, the algorithm in an error-free case and their properties relating to semi-scalar product and semi-norm. Then, we try to construct suitable inequality (show as in (\ref{ineqality for lemma4})) to get that $\textstyle\sum_{n\in\mathbb{N}} \|p_{n-1} \!-\! \tilde{p}_n\|^2 < +\infty$ by using Lemma \ref{lem 2.2, Combettes&Vu}. After that we build up an inequality to assure that the sequence $(x_n)_{n\in\mathbb{N}}$ is $|\cdot|^2$-quasi-Fejer monotone with respect to the  target set $zer(A+B)$ relative to $(U_n^{-1})_{n\in\mathbb{N}}$ and later we obtain that $\textstyle\sum_{n\in\mathbb{N}} \|p_{n-1} \!-\! \tilde{p}_n\|^2 < +\infty$. Therefore, (i) can be shown with the assistance of above bounded summable results; consequently, the quasi-Fejer monotone setting together with Theorem \ref{Thm 3.3 (10)} demonstrates (ii) as required.

Now let us show the whole proof here. It follows from Lemma (\ref{lem3.7(11)}) that the  sequences $(x_n)_{n\in\mathbb{N}}, (y_n)_{n\in\mathbb{N}}, (p_n)_{n\in\mathbb{N}}$ and $(q_n)_{n\in\mathbb{N}}$ are well defined. From (\ref{ConditionA}), we obtain that
\begin{align*}
(\forall x\in\mathcal{H})\;\inner {U_nx}{x} \leq \|U_nx\| \|x\| \leq \|U_n\| \|x\|^2 \leq \mu \|x\|^2 =\inner {\mu x}{x} \;\text{implies that}\; U_n \preccurlyeq \mu Id,
\end{align*}
and since $U_n\in \mathcal{P}_{\alpha}(\mathcal{H})$, then $U_n \succcurlyeq \alpha Id$. Hence we have that  

\begin{align}\label{ConditionAA}
\begin{cases}
\mu Id \succcurlyeq  U_n \succcurlyeq  \alpha Id, \\
\alpha^{-1} Id \succcurlyeq U_n^{-1} \succcurlyeq \mu^{-1} Id, \quad \text{by Lemma \ref{lem2.1(10)}}
\end{cases}
\end{align}

For all $g_n\in\mathcal{H}$, $n\in\mathbb{N}$,
\begin{align*}
\|g_n\|_{U_n^{-1}}=\sqrt{\inner {g_n}{U_n^{-1}g_n}} \leq \sqrt{\inner{g_n}{\alpha^{-1}Id g_n}}=\|g_n\|\sqrt{\frac{1}{\alpha}},
\end{align*}
and
\begin{align*}
\|g_n\|_{U_n^{-1}}=\sqrt{\inner {g_n}{U_n^{-1}g_n}} \geq \sqrt{\inner{g_n}{\mu^{-1}Id g_n}}=\|g_n\|\sqrt{\frac{1}{\mu}},
\end{align*}
Thus we have that $\sqrt{\frac{1}{\mu}}\|g_n\|\leq\|g_n\|_{U_n^{-1}} \leq \|g_n\|\sqrt{\frac{1}{\alpha}}.$
This means that
\begin{align}\label{sum|g|_U^-1 <infty}
\sum\limits_{n\in\mathbb{N}}\|g_n\|<+\infty\quad\Leftrightarrow\quad \sum\limits_{n\in\mathbb{N}}\|g_n\|_{U_n^{-1}} <+\infty.
\end{align}
Similarly, we also have that
\begin{small}
\begin{align*}
(\forall g_n\in\mathcal{H})\quad \sqrt{\alpha}\|g_n\|=\sqrt{\inner{\alpha g_n}{g_n}}\leq \sqrt{\inner{U_n g_n}{g_n}}=\|g_n\|_{U_{n}}=\sqrt{\inner{U_n g_n}{g_n}}\leq\sqrt{\inner{\mu g_n}{g_n}}=\|g_n\|\sqrt{\mu}.
\end{align*}
\end{small}
and then
\begin{align}\label{sum|g|_U <infty}
(\forall g_n\in\mathcal{H})\quad \sum\limits_{n\in\mathbb{N}}\|g_n\|<+\infty\quad\Leftrightarrow\quad \sum\limits_{n\in\mathbb{N}}\|g_n\|_{U_n} <+\infty .
\end{align}
Let us set 
\begin{align}\label{set_para2}
(\forall n\in\mathbb{N})\quad
\begin{cases}
\tilde{y}_n = x_n -\gamma_n U_n B(p_{n-1})\\
\tilde{p}_n = J_{\gamma_n U_n A}(\tilde{y}_n)\\
\tilde{q}_n = \tilde{p}_n-\gamma_nU_nB(\tilde{p}_n)\\
\tilde{x}_{n+1} = x_n -\tilde{y}_n + \tilde{q}_n
\end{cases}
and
\begin{cases}
u_n = \gamma_n^{-1}U_n^{-1}(x_n-\tilde{p}_n)+B(\tilde{p}_n)-B(p_{n-1})\\
e_n = \tilde{x}_{n+1}-x_{n+1} = y_n -q_n -\tilde{y}_n+\tilde{q}_n.\\
\end{cases}
\end{align}
Since $\tilde{p}_n = J_{\gamma_n U_n A}(\tilde{y}_n)$, then $\tilde{y}_n \in \tilde{p}_n + \gamma_n U_n A(\tilde{p}_n)$ and therefore
\begin{align}\label{belongtoA(pn)}
\gamma_n^{-1} U_n^{-1} (\tilde{y}_n-\tilde{p}_n)\in A(\tilde{p}_n).
\end{align}
From (\ref{set_para2}) and (\ref{belongtoA(pn)}), we have that
\begin{align}\label{u_n}
(\forall n\in\mathbb{N})\quad u_n&=\gamma_n^{-1}U_n^{-1}\left(x_n-\gamma_nU_nB(p_{n-1})-\tilde{p}_n\right)+B(p_{n-1})+B(\tilde{p}_n)-B(p_{n-1})\notag \\
&= \gamma_n^{-1}U_n^{-1}(\tilde{y}_n-\tilde{p}_n)+B(\tilde{p}_n)\in A(\tilde{p}_n)+B(\tilde{p}_n)=(A+B)(\tilde{p}_n).
\end{align}
Since for all $ x\in\mathcal{H}$, we observe that 
\begin{align} \label{|U_n(x)|=|x|_(U_n)}
\|U_n x\|_{U_n^{-1}}=\sqrt{\inner{U_n^{-1}U_nx}{U_nx}}=\sqrt{\inner{x}{U_nx}}=\|x\|_{U_n}.
\end{align} 
Applying (\ref{set_para1}), (\ref{set_para2}), (\ref{ConditionAA}), (\ref{|U_n(x)|=|x|_(U_n)}),  Lemma \ref{lem3.7(11)} and the $\beta$-Lipschitz continuity of $B$ yield
\begin{align*} 
\|y_n-\tilde{y}_n\|_{U_n^{-1}}
&=\gamma_n\|U_na_n\|_{U_n^{-1}}
=\gamma_n\|a_n\|_{U_n} \leq \lambda\|a_n\|_{U_n},
\end{align*}
and
\begin{align}\label{|pn-tilde(p)n| bounded from above}
\|p_n-\tilde{p}_n\|_{U_n^{-1}} &= \|J_{\gamma_nU_nA}(y_n)+b_n - J_{\gamma_nU_nA}(\tilde{y}_n)\|_{U_n^{-1}} \notag\\
&\leq \|J_{\gamma_nU_nA}(y_n)-J_{\gamma_nU_nA}(\tilde{y}_n)\|_{U_n^{-1}} +\|b_n\|_{U_n^{-1}} \notag\\
&\leq \|y_n-\tilde{y}_n
\|_{U_n^{-1}}+\|b_n\|_{U_n^{-1}} \notag\\
&\leq \lambda\|a_n\|_{U_n}+\|b_n\|_{U_n^{-1}},
\end{align}
and
\begin{align}\label{|qn-tilde(q)_n}
\|q_n-\tilde{q}_n\|_{U_n^{-1}} &= \|p_n-\gamma_nU_n \left( B(p_n)+c_n\right)-\tilde{p}_n+\gamma_nU_nB(\tilde{p}_n)\|_{U_n^{-1}} \notag\\
&\leq \|p_n-\tilde{p}_n\|_{U_n^{-1}} +\|\gamma_nU_nB(\tilde{p}_n)-\gamma_nU_nB(p_n)\|_{U_n^{-1}}+\gamma_n\|U_n c_n\|_{U_n^{-1}}, 
\end{align}
then, we consider
\begin{small}
\begin{align*}
\|\gamma_nU_n \left(B(\tilde{p}_n)-B(p_n)\right)\|^2_{U_n^{-1}} &=\|\gamma_n\left(B(\tilde{p}_n)-B(p_n)\right)\|^2_{U_n}\\
&=\gamma^2_n\inner{B(\tilde{p}_n)-B(p_n)}{U_nB(\tilde{p}_n)-U_nB(p_n)}\\
&\leq\gamma^2_n\|U_n\|\|B(\tilde{p}_n)-B(p_n)\|^2\\
&\leq \gamma_n^2 \mu\beta^2\|\tilde{p}_n-p_n\|^2 \;[\text{since}\; \mu =\sup_{n\in\mathbb{N}}\|U_n\|<+\infty, \text{$\beta$-Lipschitz continuity of $B$}]\\
&\leq \gamma_n^2 \mu^2 \beta^2\|\tilde{p}_n-p_n\|^2_{U_n^{-1}}\; [ \text{since}\; \mu^{-1}\|g_n\|^2\leq\|g_n\|^2_{U_n^{-1}}, \forall g_n\in\mathcal{H}]\\
&\leq \|\tilde{p}_n-p_n\|^2_{U_n^{-1}}\; \left[ \text{since}\; \gamma_n \leq \lambda  < \frac{1}{ \sqrt{10}\mu\beta}\leq\frac{1}{\beta\mu}\right].
\end{align*}
\end{small}
Then, it follows from (\ref{|U_n(x)|=|x|_(U_n)}), (\ref{|qn-tilde(q)_n}) and (\ref{|pn-tilde(p)n| bounded from above}) that
\begin{small}
\begin{align*}
\|q_n-\tilde{q}_n\|_{U_n^{-1}} &\leq 2\|p_n-\tilde{p}_n\|_{U_n^{-1}}+\gamma_n\|U_n c_n\|_{U_n^{-1}} \notag\\
&\leq 2 \left[ \|b_n\|_{U_n^{-1}}+\lambda\|a_n\|_{U_n} \right]+\gamma_n\|c_n\|_{U_n}\; \notag \\
&\leq 2\left[ \|b_n\|_{U_n^{-1}}+\lambda\|a_n\|_{U_n} \right] + \lambda \|c_n\|_{U_n} \left[\text{since}\; \gamma_n \leq \lambda, \forall n\in\mathbb{N} \right].
\end{align*}
\end{small}

Hence, we have that
\begin{align}\label{|yn-yilde(y)n|,|pn-tilde(p)n|,|qn-tilde(q)n|}
(\forall n\in\mathbb{N})\quad\begin{cases}
\|y_n-\tilde{y}_n\|_{U_n^{-1}}\leq \lambda\|a_n\|_{U_n},\\
\| p_n-\tilde{p}_n\|_{U_n^{-1}}\leq \|b_n\|_{U_n^{-1}}+\lambda\|a_n\|_{U_n},\\
\|q_n -\tilde{q}_n\|_{U_n^{-1}}\leq 2\left[ \|b_n\|_{U_n^{-1}}+\lambda\|a_n\|_{U_n} \right] + \lambda \|c_n\|_{U_n}.
\end{cases}
\end{align}
Since $(a_n)_{n\in\mathbb{N}},(b_n)_{n\in\mathbb{N}}$ and $(c_n)_{n\in\mathbb{N}}$ are absolutely summable sequences in $\mathcal{H}$, we derive from (\ref{sum|g|_U^-1 <infty}), (\ref{sum|g|_U <infty}), (\ref{set_para2}) and (\ref{|yn-yilde(y)n|,|pn-tilde(p)n|,|qn-tilde(q)n|})  that
\begin{align}\label{the sum of |yn-yilde(y)n|,|pn-tilde(p)n|,|qn-tilde(q)n| bounded}
\begin{cases}
\sum\limits_{n\in\mathbb{N}} \|y_n-\tilde{y}_n\| <+\infty \quad &\text{and} \quad \sum\limits_{n\in\mathbb{N}} \|y_n-\tilde{y}_n\|_{U_n^{-1}} <+\infty,\\
\sum\limits_{n\in\mathbb{N}} \|p_n-\tilde{p}_n\| <+\infty \quad &\text{and} \quad \sum\limits_{n\in\mathbb{N}} \|p_n-\tilde{p}_n\|_{U_n^{-1}} <+\infty,\\
\sum\limits_{n\in\mathbb{N}} \|q_n-\tilde{q}_n\| <+\infty \quad &\text{and} \quad \sum\limits_{n\in\mathbb{N}} \|q_n-\tilde{q}_n\|_{U_n^{-1}} <+\infty,\\
\end{cases}
\end{align}
Follows from (\ref{set_para1}), (\ref{set_para2}) and (\ref{|yn-yilde(y)n|,|pn-tilde(p)n|,|qn-tilde(q)n|}), we can derive that
\begin{align}\label{|e_n| bounded}
(\forall n\in\mathbb{N})\; \|e_n\|&=\|\tilde{x}_{n+1}-x_{n+1}\| \notag\\
&\leq\|y_n-q_n-\tilde{y}_n-\tilde{q}_n\| \notag\\
&\leq \|y_n-\tilde{y}_n\|+\|q_n-\tilde{q}_n\| \notag\\
&\leq
\left( \lambda \|a_n\|_{U_n} \right) + 2\left( \|b_n\|_{U_n^{-1}} + \lambda \|a_n\|_{U_n} \right) + \lambda \|c_n\|_{U_n}.
\end{align} 
Since $(a_n)_{n\in\mathbb{N}},(b_n)_{n\in\mathbb{N}}$ and $(c_n)_{n\in\mathbb{N}}$ are absolutely summable sequences in $\mathcal{H}$, we derive from (\ref{sum|g|_U^-1 <infty}), (\ref{sum|g|_U <infty}), (\ref{|e_n|  bounded}) and (\ref{the sum of |yn-yilde(y)n|,|pn-tilde(p)n|,|qn-tilde(q)n| bounded}) that $\sum_{n\in\mathbb{N}}\|e_n\|<+\infty$ and 
\begin{align}\label{sum|e_n| bounded}
\sum\limits_{n\in\mathbb{N}} \|e_n\| < +\infty \Leftrightarrow \sum\limits_{n\in\mathbb{N}} \|e_n\|_{U_n^{-1}} < +\infty \Leftrightarrow \sum\limits_{n\in\mathbb{N}} \|e_n\|_{U_n} < +\infty.
\end{align}

Now, we let $x\in zer(A+B)$.
Then, for every $n\in\mathbb{N}$, $(x,-\gamma_nU_nB(x))\in Gra (\gamma_nU_nA)$ [because  $x\in zer(A+B)\Leftrightarrow 0 \in \gamma_nU_n A(x)+\gamma_n U_n B(x)\Leftrightarrow -\gamma_nU_nB(x)\in \gamma_nU_nA(x)$] and (\ref{set_para2}) yields $(\tilde{p}_n,\tilde{y}_n-\tilde{p}_n)\in Gra (\gamma_nU_nA)$ 
[see, Equation (\ref{belongtoA(pn)})].
Hence by monotonicity of $U_nA$ with respect to the scalar product $\inner{\cdot}{\cdot}_{U_n^{-1}}$ in Lemma \ref{lem3.7(11)} (i), we obtain that
\begin{align*}
\inner{\tilde{p}_n-x}{\tilde{p}_n-\tilde{y}_n-\gamma_nU_nB(x)}_{U_n^{-1}}\leq 0,
\end{align*}
moreover, by monotonicity of $\gamma_nU_nB$ with respect to the scalar product  $\inner{\cdot}{\cdot}_{U_n^{-1}}$, we also have
\begin{align*}
\inner{\tilde{p}_n-x}{\gamma_nU_nB(x)-\gamma_nU_nB(\tilde{p}_n)}_{U_n^{-1}} \leq 0.
\end{align*}
By the last two inequalities, we obtain
\begin{align}\label{<pt_n-x,pt_n-yt_n-gam_nUnB(pt_n)_Un^-1>}
(\forall n\in\mathbb{N})\quad \inner{\tilde{p}_n-x}{\tilde{p}_n-\tilde{y}_n-\gamma_nU_nB(\tilde{p}_n)}_{U_n^{-1}} \leq 0.
\end{align}
In turn, we derive from (\ref{set_para2}) and (\ref{<pt_n-x,pt_n-yt_n-gam_nUnB(pt_n)_Un^-1>}) that
\begin{align}\label{2gamma_n<,>}
(\forall n\in\mathbb{N})\;\; 2\gamma_n\inner{\tilde{p}_n - x}{U_nB(p_{n-1})-U_nB(\tilde{p}_n)}_{U_n^{-1}}&=2\inner{\tilde{p}_n-x}{\tilde{p}_n-\tilde{y}_n-\gamma_nU_nB(\tilde{p}_n)}_{U_n^{-1}}\notag\\&+2\inner{\tilde{p}_n-x}{\gamma_nU_nB(p_{n-1})+\tilde{y}_n-\tilde{p}_n}_{U_n^{-1}}\notag\\
&\leq 2\inner{\tilde{p}_n-x}{\gamma_nU_nB(p_{n-1})+\tilde{y}_n-\tilde{p}_n}_{U_n^{-1}}\notag\\
&= 2\inner{\tilde{p}_n-x}{x_n-\tilde{p}_n}_{U_n^{-1}}\notag\\
&=\|x_n-x\|^2_{U_n^{-1}}-\|\tilde{p}_n-x\|^2_{U_n^{-1}}-\|x_n-\tilde{p}_n\|^2_{U_n^{-1}}.
\end{align}
Next, using (\ref{set_para2}), (\ref{2gamma_n<,>}), (\ref{|U_n(x)|=|x|_(U_n)}) , (\ref{ConditionAA}), the $\beta$-Lipschitz continuity of $B$ and Lemma \ref{lem2.1(10)}, for every $n\in\mathbb{N}$, we obtain
\begin{small}
\begin{align}\label{|tilde(x)_n+1 - x|^2_U_n(-1)}
\|\tilde{x}_{n+1}-x\|^2_{U_n^{-1}}&=\|\tilde{q}_n+x_n-\tilde{y}_n-x\|^2_{U_n^{-1}}\notag\\
&=\|(\tilde{p}_n-x)+\gamma_nU_n(B(p_{n-1})-B(\tilde{p}_n))\|^2_{U_n^{-1}}\notag\\
&=\|\tilde{p}_n-x\|^2_{U_n^{-1}}+2\gamma_n\inner{\tilde{p}_n-x}{U_n(B(p_{n-1})-B(\tilde{p}_n))}_{U_n^{-1}}+\gamma_n^2\|U_n(B(p_{n-1})-B(\tilde{p}_n))\|^2_{U_n^{-1}}\notag\\
&\leq \|\tilde{p}_n-x\|^2_{U_n^{-1}} +\left[ \|x_n-x\|^2_{U_n^{-1}}-\|\tilde{p}_n-x\|^2_{U_n^{-1}}-\|x_n-\tilde{p}_n\|^2_{U_n^{-1}} \right] +\gamma_n^2\|B(p_{n-1})-B(\tilde{p}_n)\|^2_{U_n}\notag\\
&\leq\|x_n-x\|^2_{U_n^{-1}}-\|x_n-\tilde{p}_n\|^2_{U_n^{-1}}+\gamma^2_n\beta^2\mu\|p_{n-1}-\tilde{p}_n\|^2 \; [\text{since}\; \alpha Id \preccurlyeq U_n \preccurlyeq \mu Id \;\text{in}\; (\ref{ConditionAA})]\notag\\
&\leq \|x_n-x\|^2_{U_n^{-1}}-\mu^{-1}\|x_n-\tilde{p}_n\|^2 +\gamma^2_n\beta^2\mu\|p_{n-1}-\tilde{p}_n\|^2 \; [\text{since}\; \alpha^{-1} Id \succcurlyeq U_n^{-1} \succcurlyeq \mu^{-1} Id ].
\end{align}
\end{small}

\noindent By Parallelogram law, we have  $2\|x_n-\tilde{p}_n\|^2+2\|x_n-p_{n-1}\|^2=\|p_{n-1}-\tilde{p}_n\|^2+\|(x_n-\tilde{p}_n)+(x_n-p_{n-1})\|^2$, then
$\|p_{n-1}-\tilde{p}_n\|^2\leq 2\|x_n-\tilde{p}_n\|^2+2\|x_n-p_{n-1}\|^2$ and so $\|x_n-\tilde{p}_{n}\|^2\geq  -\|x_n-p_{n-1}\|^2+\frac{1}{2}\|p_{n-1}-\tilde{p}_n\|^2$. Hence
\begin{align}\label{-|xn-tilde(p)_n|^2}
-\|x_n-\tilde{p}_{n}\|^2\leq  \|x_n-p_{n-1}\|^2-\frac{1}{2}\|p_{n-1}-\tilde{p}_n\|^2.
\end{align}
Now we follows from (\ref{|tilde(x)_n+1 - x|^2_U_n(-1)}) and (\ref{-|xn-tilde(p)_n|^2}) that
\begin{align}
\|\tilde{x}_{n+1}-x\|^2_{U_n^{-1}} &\leq \|x_n-x\|^2_{U_n^{-1}}-\mu^{-1}\|x_n-\tilde{p}_n\|^2 +\gamma^2_n\beta^2\mu\|p_{n-1}-\tilde{p}_n\|^2 \notag\\
&\leq \|x_n-x\|^2_{U_n^{-1}}+\mu^{-1}\left[ \|x_n-p_{n-1} \|^2-\frac{1}{2}\|p_{n-1}-\tilde{p}_n\|^2\right] + \gamma^2_n\beta^2\mu\|p_{n-1}-\tilde{p}_n\|^2 \notag\\
&= \|x_n-x\|^2_{U_n^{-1}}+\mu^{-1} \|x_n-p_{n-1} \|^2 + \left( \gamma_n^2\beta^2\mu - \frac{1}{2\mu}\right) \|p_{n-1}-\tilde{p}_n\|^2.
\end{align} 
Then we obtain that 
\begin{align} \label{Before Telescope 1} 
\|\tilde{x}_{n+1}-x\|^2_{U_n^{-1}}+\left( \frac{1}{2\mu}- \gamma_n^2\beta^2\mu \right) \|p_{n-1}-\tilde{p}_n\|^2 \leq  \|x_n-x\|^2_{U_n^{-1}}+\mu^{-1} \|x_n-p_{n-1} \|^2.  
\end{align}
Since (\ref{set_para1}) gives us that for all $n\in\mathbb{N}$, $x_{n+1} = x_n-y_n+q_n = \gamma_nU_n\left( B(p_{n-1}) + a_n \right) + p_n-\gamma_nU_n\left(B(p_n)+c_n\right)$, then we have $x_n = \gamma_{n-1}U_{n-1}\left( B(p_{n-2}) + a_{n-1} \right) + p_{n-1}-\gamma_{n-1}U_{n-1}\left(B(p_{n-1})+c_{n-1}\right)$. Therefore 
\begin{align}\label{|x_n-p_(n-1)|}
\|x_n-p_{n-1}\| 
&\leq \gamma_{n-1}\|U_{n-1}B(p_{n-2})-U_{n-1}B(p_{n-1})\|+\gamma_{n-1}\|U_{n-1}\left(a_{n-1}-c_{n-1}\right) \|  \notag\\
&\leq \gamma_{n-1} \mu \beta \|p_{n-2}-p_{n-1} \|+\gamma_{n-1} \mu \left( \|a_{n-1}\| + \|c_{n-1}\| \right). 
\end{align}
It follows from (\ref{Before Telescope 1}), (\ref{|x_n-p_(n-1)|}) and Cauchy-Schwarz inequality  that $(\forall n\in\mathbb{N})$,

\begin{footnotesize}
\begin{align}\label{Before Telescope 2 not replace zn yet}
\|\tilde{x}_{n+1}-x\|^2_{U_n^{-1}}+\left( \frac{1}{2\mu}- \gamma_n^2\beta^2\mu \right) \|p_{n-1}-\tilde{p}_n\|^2 
&\leq  \|x_n-x\|^2_{U_n^{-1}}+\mu^{-1} \left[ \gamma_{n-1} \mu\beta \|p_{n-2}-p_{n-1} \|+\gamma_{n-1} \mu \left(\|a_{n-1}\|+\|c_{n-1}\|\right)  \right]^2 \notag\\
&\leq \|x_n-x\|^2_{U_n^{-1}}+\mu^{-1} \left[ 2(\gamma_{n-1}\mu\beta)^2 \|p_{n-2}-p_{n-1}\|^2+2(\gamma_{n-1}\mu)^2(\|a_{n-1}\|+\|c_{n-1}\|)^2 \right] \notag \\
&\leq \|x_n-x\|^2_{U_n^{-1}}+ 2\gamma_{n-1}^2\beta^2\mu \|p_{n-2}-p_{n-1}\|^2 + 2\gamma_{n-1}^2\mu (\|a_{n-1}\|+\|c_{n-1}\|)^2 \notag\\
&\leq \|x_n-x\|^2_{U_n^{-1}}+2 \gamma_{n-1}^2\beta^2\mu \left[2(\|p_{n-2}-\tilde{p}_{n-1}\|^2 +\|\tilde{p}_{n-1}-p_{n-1}\|^2) \right
]\notag\\
&+2  \gamma_{n-1}^2\mu (\|a_{n-1}\|+\|c_{n-1}\|)^2 \notag\\
&\leq \|x_n-x\|^2_{U_n^{-1}}+4\gamma_{n-1}^{2}\mu\beta^2\|p_{n-2}-\tilde{p}_{n-1}\|^2+4\gamma_{n-1}^{2}\mu\beta^2\|\tilde{p}_{n-1}-p_{n-1}\|^2 \notag \\
&+2\gamma_{n-1}^2\mu(\|a_{n-1}\|+\|c_{n-1}\|)^2.
\end{align}
\end{footnotesize}
Let $z_n= \tilde{x}_{n+1}-x = x_n-\tilde{y}_n+\tilde{q}_n-x$ and $M_n = \frac{1}{2\mu}-\gamma^2_n\beta^2\mu >0$ (Since $\gamma_n < \frac{1}{\sqrt{2}\beta\mu},\;\forall n\in\mathbb{N}$), then we derive from (\ref{Before Telescope 2 not replace zn yet}) that $(\forall n\in\mathbb{N})$,
\begin{align}\label{|zn|+Mn|p_(n-1)-tilde(p)_n|}
\|z_n\|^2_{U_n^{-1}} + M_n\|p_{n-1}-\tilde{p}_n\|^2 &= \|\tilde{x}_{n+1}-x\|^2_{U_n^{-1}}+\left( \frac{1}{2\mu}- \gamma_n^2\beta^2\mu \right) \|p_{n-1}-\tilde{p}_n\|^2 \notag\\
&\leq \|x_n-x\|^2_{U_n^{-1}}+4\gamma_{n-1}^{2}\mu\beta^2\|p_{n-2}-\tilde{p}_{n-1}\|^2 \notag \\
&+4\gamma_{n-1}^{2}\mu\beta^2 \|\tilde{p}_{n-1}-p_{n-1}\|^2+2\gamma_{n-1}^2\mu(\|a_{n-1}\|+\|c_{n-1}\|)^2.
\end{align}
Applying (\ref{set_para1}) and (\ref{set_para2}), we obtain that $x_{n+1}-x = \tilde{x}_{n+1}- x - \tilde{x}_{n+1}+x_{n+1}= z_n -e_n$.
Then we get that
\begin{align}\label{|x_n+1-x|_U_n+1^-1}
\|x_{n+1}-x\|_{U_{n}^{-1}}^2 = \|z_n\|_{U_{n}^{-1}}^2 -2\inner{z_n}{e_n}_{U_{n}^{-1}} + \|e_n\|_{U_{n}^{-1}}^2.
\end{align}
From (\ref{ConditionA}), we know that $(\forall n\in N) \; (1+\eta_n)U_{n+1} \succcurlyeq U_n$. It follows from (\ref{|x_n+1-x|_U_n+1^-1}) that  
\begin{align}\label{square of |x_n+1-x|_U_n+1^-1}
\|x_{n+1}-x\|_{U_{n+1}^{-1}}^2 
&\leq (1+\eta_n)\|x_{n+1}-x\|_{U_n^{-1}}^2 \notag\\
&=(1+\eta_n) \left(\|z_n\|_{U_{n}^{-1}}^2 -2\inner{z_n}{e_n}_{U_{n}^{-1}} + \|e_n\|_{U_{n }^{-1}}^2\right).
\end{align}
By using (\ref{square of |x_n+1-x|_U_n+1^-1}) and  (\ref{|zn|+Mn|p_(n-1)-tilde(p)_n|}) yield
\begin{align}\label{|x_n+1 - x|^2_n+1 + Mm|p_(n-1)-tilde(p)_n|^2< something}
\|x_{n+1}-x\|_{U_{n+1}^{-1}}^2+ M_n\|p_{n-1}-\tilde{p}_n\|^2 
&\leq \left(\|z_n\|_{U_{n}^{-1}}^2 +  M_n\|p_{n-1}-\tilde{p}_n\|^2\right)+ \eta_n \|z_n\|_{U_{n}^{-1}}^2 \notag\\
&+(1+\eta_n) (-2\inner{z_n}{e_n}_{U_{n}^{-1}}) +(1+\eta_n) \|e_n\|_{U_{n}^{-1}}^2 \notag\\
&\leq \big[ \|x_n-x\|^2_{U_n^{-1}}+4\gamma_{n-1}^{2}\mu\beta^2\|p_{n-2}-\tilde{p}_{n-1}\|^2 
\notag \\
&+4\gamma_{n-1}^{2}\mu\beta^2\|\tilde{p}_{n-1}-p_{n-1}\|^2+2\gamma_{n-1}^2\mu(\|a_{n-1}\|+\|c_{n-1}\|)^2 \big]
\notag \\
&+\eta_n \big[ \|x_n-x\|^2_{U_n^{-1}}+4\gamma_{n-1}^{2}\mu\beta^2\|p_{n-2}-\tilde{p}_{n-1}\|^2 
\notag\\
&+4\gamma_{n-1}^{2}\mu\beta^2\|\tilde{p}_{n-1}-p_{n-1}\|^2+2\gamma_{n-1}^2\mu(\|a_{n-1}\|+\|c_{n-1}\|)^2\big]
\notag\\
&+(1+\eta_n) (-2\inner{z_n}{e_n}_{U_{n}^{-1}}) +(1+\eta_n) \|e_n\|_{U_{n}^{-1}}^2 
\notag\\
&=(1+\eta_n)\|x_n-x\|^2_{U_n^{-1}}+(1+\eta_n)4\gamma_{n-1}^{2}\mu\beta^2\|p_{n-2}-\tilde{p}_{n-1}\|^2
\notag\\
&+(1+\eta_n)4\gamma_{n-1}^{2}\mu\beta^2\|\tilde{p}_{n-1}-p_{n-1}\|^2
\notag\\
&+(1+\eta_n)2\gamma_{n-1}^2\mu(\|a_{n-1}\|+\|c_{n-1}\|)^2
\notag\\
&+(1+\eta_n) (-2\inner{z_n}{e_n}_{U_{n}^{-1}}) +(1+\eta_n) \|e_n\|_{U_{n}^{-1}}^2.  
\end{align}
Now, from (\ref{|zn|+Mn|p_(n-1)-tilde(p)_n|}), we obtain $(\forall n\in\mathbb{N})$,
\begin{align}\label{-2<zn,en>}
-2\inner{z_n}{e_n}_{U_n^{-1}} &\leq 2\|z_n\|_{U_{n}^{-1}} \|e_n\|_{U_{n}^{-1}} \notag\\
&\leq(\|z_n\|_{U_{n}^{-1}}^2+1) \|e_n\|_{U_{n}^{-1}} \notag \\
&\leq \big[ \|x_n-x\|^2_{U_n^{-1}}+4\gamma_{n-1}^{2}\mu\beta^2\|p_{n-2}-\tilde{p}_{n-1}\|^2 
+4\gamma_{n-1}^{2}\mu\beta^2\|\tilde{p}_{n-1}-p_{n-1}\|^2\notag \\&+2\gamma_{n-1}^2\mu(\|a_{n-1}\|+\|c_{n-1}\|)^2 +1 \big] \|e_n\|_{U_{n}^{-1}} 
\notag \\
&=\|e_n\|_{U_{n}^{-1}} \|x_n-x\|^2_{U_n^{-1}}+\|e_n\|_{U_{n}^{-1}} 4\gamma_{n-1}^{2}\mu\beta^2\|p_{n-2}-\tilde{p}_{n-1}\|^2 
\notag \\
&+\|e_n\|_{U_{n}^{-1}} 4\gamma_{n-1}^{2}\mu\beta^2\|\tilde{p}_{n-1}-p_{n-1}\|^2
\notag \\&+\|e_n\|_{U_{n}^{-1}} 2\gamma_{n-1}^2\mu(\|a_{n-1}\|+\|c_{n-1}\|)^2 +\|e_n\|_{U_{n}^{-1}}. 
\end{align}
Consider (\ref{|x_n+1 - x|^2_n+1 + Mm|p_(n-1)-tilde(p)_n|^2< something}) together with (\ref{-2<zn,en>}), then $(\forall n\in\mathbb{N})$,
\begin{small}
\begin{align}\label{Before Telescope 3 before set Dn}
 \|x_{n+1}-x\|^2_{U_{n+1}^{-1}} + M_n\|p_{n-1}-\tilde{p}_n\|^2 
&\leq (1+\eta_n)\|x_n-x\|^2_{U_n^{-1}}+(1+\eta_n)4\gamma_{n-1}^{2}\mu\beta^2\|p_{n-2}-\tilde{p}_{n-1}\|^2
\notag\\
&+(1+\eta_n)4\gamma_{n-1}^{2}\mu\beta^2\|\tilde{p}_{n-1}-p_{n-1}\|^2
\notag\\
&+(1+\eta_n)2\gamma_{n-1}^2\mu(\|a_{n-1}\|+\|c_{n-1}\|)^2
\notag\\
&+(1+\eta_n) \bigg[ \|e_n\|_{U_{n}^{-1}} \|x_n-x\|^2_{U_n^{-1}}+\|e_n\|_{U_{n}^{-1}} 4\gamma_{n-1}^{2}\mu\beta^2\|p_{n-2}-\tilde{p}_{n-1}\|^2 
\notag \\
&+\|e_n\|_{U_{n}^{-1}} 4\gamma_{n-1}^{2}\mu\beta^2\|\tilde{p}_{n-1}-p_{n-1}\|^2
\notag \\
&+\|e_n\|_{U_{n}^{-1}} 2\gamma_{n-1}^2\mu (\|a_{n-1}\|+\|c_{n-1}\|)^2 +\|e_n\|_{U_{n}^{-1}} 
\bigg]
\notag \\
&+(1+\eta_n) \|e_n\|_{U_{n}^{-1}}^2
\notag \\
&= (1+\eta_n) \bigg[ \|x_n-x\|^2_{U_n^{-1}}+4\gamma_{n-1}^{2}\mu\beta^2\|p_{n-2}-\tilde{p}_{n-1}\|^2
\notag\\
&+4\gamma_{n-1}^{2}\mu\beta^2\|\tilde{p}_{n-1}-p_{n-1}\|^2
+2\gamma_{n-1}^2\mu(\|a_{n-1}\|+\|c_{n-1}\|)^2 \bigg]
\notag\\
&+(1+\eta_n)\|e_n\|_{U_{n}^{-1}} \bigg[ \|x_n-x\|^2_{U_n^{-1}}+ 4\gamma_{n-1}^{2}\mu\|p_{n-2}-\tilde{p}_{n-1}\|^2 
\notag \\
&+4\gamma_{n-1}^{2}\mu\beta^2\|\tilde{p}_{n-1}-p_{n-1}\|^2
+ 2\gamma_{n-1}^2\mu(\|a_{n-1}\|+\|c_{n-1}\|)^2 +1
\bigg]
\notag \\
&+(1+\eta_n) \|e_n\|_{U_{n}^{-1}}^2
\notag \\
&=\left[ \left(1+\eta_n\right)\left(1+\|e_n\|_{U_{n}^{-1}}\right)  \right ]   \left[ \|x_n-x\|^2_{U_n^{-1}}+4\gamma_{n-1}^{2}\mu\beta^2\|p_{n-2}-\tilde{p}_{n-1}\|^2  \right ]
\notag\\
&+\left[ \left(1+\eta_n\right)\left(1+\|e_n\|_{U_{n}^{-1}}\right)  \right ]  \left[ 4\gamma_{n-1}^{2}\mu\beta^2\|\tilde{p}_{n-1}-p_{n-1}\|^2 \right ]
\notag \\
&+\left[ \left(1+\eta_n\right)\left(1+\|e_n\|_{U_{n}^{-1}}\right)  \right ]  \left[ 2\gamma_{n-1}^2\mu(\|a_{n-1}\|+\|c_{n-1}\|)^2 \right ]
\notag \\
&+2(1+\eta_n) \|e_n\|_{U_{n}^{-1}}^2
\notag \\
&=\left[ \!1\!+\!\left(\eta_n\!+\!\|e_n\|_{U_{n}^{-1}}\!+\!\eta_n \|e_n\|_{U_{n}^{-1}}\right)  \right ]   \left[\! \|x_n\!-\!x\|^2_{U_n^{-1}}\!+\!4\gamma_{n-1}^{2}\mu\beta^2\|p_{n-2}\!-\!\tilde{p}_{n-1}\|^2\!  \right ]
\notag\\
&+\left[ 1+\left(\eta_n+\|e_n\|_{U_{n}^{-1}}+\eta_n \|e_n\|_{U_{n}^{-1}}\right)  \right ]  \left[ 4\gamma_{n-1}^{2}\mu\beta^2\|\tilde{p}_{n-1}-p_{n-1}\|^2 \right ]
\notag \\
&+\left[ 1+\left(\eta_n+\|e_n\|_{U_{n}^{-1}}+\eta_n \|e_n\|_{U_{n}^{-1}}\right)  \right ]  \left[ 2\gamma_{n-1}^2\mu(\|a_{n-1}\|+\|c_{n-1}\|)^2 \right ]
\notag \\
&+2(1+\eta_n) \|e_n\|_{U_{n}^{-1}}^2.
\end{align}
\end{small}
Because $(\forall n\in\mathbb{N})\; \gamma_n\leq\lambda < \frac{1}{\sqrt{10}\mu\beta}$ $\left(<\frac{1}{\sqrt{2}\beta\mu}\right)$.
Then $\lambda^2 <\frac{1}{10\mu^2\beta^2}$ and so $5\lambda^2\mu\beta^2 < \frac{1}{2\mu}$. Thus $4\gamma^2_{n-1}\mu\beta^2+\gamma^2_{n-1}\mu\beta^2 < 4\lambda^2\mu\beta^2+\lambda^2\mu\beta^2 < \frac{1}{2\mu} $  $\left(\;\text{or}\; 4\gamma^2_{n-1}\mu\beta^2+\gamma^2_{n}\mu\beta^2 < 4\lambda^2\mu\beta^2+\lambda^2\mu\beta^2 < \frac{1}{2\mu} \right)$. Therefore $4\gamma^2_{n-1}\mu\beta^2 < \frac{1}{2\mu}-\gamma^2_{n-1}\mu\beta^2= M_{n-1}. $ 

\noindent From (\ref{Before Telescope 3 before set Dn}), we let $D_n= \|x_{n+1}-x\|_{U_{n+1}^{-1}}^2 + M_n\|p_{n-1}-\tilde{p}_n\|^2$ and $D_{n-1} = \|x_{n}-x\|_{U_{n}^{-1}}^2 + M_{n-1}\|p_{n-2}-\tilde{p}_{n-1}\|^2$, then we have $(\forall n\in\mathbb{N})$
\begin{align}\label{D_n}
D_n &= \|x_{n+1}-x\|_{U_{n+1}^{-1}}^2 + M_n\|p_{n-1}-\tilde{p}_n\|^2 
\notag \\
&\leq \left[ \!1\!+\!\left(\eta_n\!+\!\|e_n\|_{U_{n}^{-1}}\!+\!\eta_n \|e_n\|_{U_{n}^{-1}}\right)  \right ]   \left[\! \|x_n\!-\!x\|^2_{U_n^{-1}}\!+\!4\gamma_{n-1}^{2}\mu\beta^2\|p_{n-2}\!-\!\tilde{p}_{n-1}\|^2\!  \right ]
\notag\\
&+\left[ 1+\left(\eta_n+\|e_n\|_{U_{n}^{-1}}+\eta_n \|e_n\|_{U_{n}^{-1}}\right)  \right ]  \left[ 4\gamma_{n-1}^{2}\mu\beta^2\|\tilde{p}_{n-1}-p_{n-1}\|^2 \right ]
\notag \\
&+\left[ 1+\left(\eta_n+\|e_n\|_{U_{n}^{-1}}+\eta_n \|e_n\|_{U_{n}^{-1}}\right)  \right ]  \left[ 2\gamma_{n-1}^2\mu(\|a_{n-1}\|+\|c_{n-1}\|)^2 \right ]
\notag \\
&+2(1+\eta_n) \|e_n\|_{U_{n}^{-1}}^2
\notag\\
&\leq \left( 1+ \tilde{\eta}_n \right) D_{n-1} +E_n,
\end{align}

where  $\tilde{\eta}_n = \eta_n+\|e_n\|_{U_{n}^{-1}}+\eta_n \|e_n\|_{U_{n}^{-1}}$, 
\begin{fleqn}[\parindent]
\begin{align}
\mbox{and}\; E_n &= \left[ 1+\tilde{\eta}_n\right ]  \left[ 4\gamma_{n-1}^{2}\mu\beta^2\|\tilde{p}_{n-1}-p_{n-1}\|^2 \right ]
\notag \\
&+\left[ 1+\tilde{\eta}_n  \right ]  \left[ 2\gamma_{n-1}^2\mu(\|a_{n-1}\|+\|c_{n-1}\|)^2 \right ]
\notag \\
&+2(1+\eta_n) \|e_n\|_{U_{n}^{-1}}^2.
\end{align}
\end{fleqn}
Since $(a_n)_{n\in\mathbb{N}},(b_n)_{n\in\mathbb{N}}$ and $(c_n)_{n\in\mathbb{N}}$ are absolutely summable sequences in $\mathcal{H}$, (\ref{the sum of |yn-yilde(y)n|,|pn-tilde(p)n|,|qn-tilde(q)n| bounded}), (\ref{|e_n| bounded}), (\ref{sum|e_n| bounded}), $(\gamma_n)_{n\in\mathbb{N}}$ is bounded and $\eta_n\in \ell_{+}^1(\mathbb{N})$, then we can conclude that
\begin{align}\label{tilde(eta_n) and sum of En bounded}
\tilde{\eta}_n\in \ell_{+}^1(\mathbb{N})  \quad \mbox{and} \quad \sum\limits_{n\in\mathbb{N}} E_n<+\infty.
\end{align} 
From (\ref{D_n}) and (\ref{tilde(eta_n) and sum of En bounded}), we know that 
\begin{align} \label{ineqality for lemma4}
D_n\leq(1+\tilde{\eta}_n) D_{n-1} +E_n\quad \mbox{with}\quad  \tilde{\eta}_n\in\ell_{+}^{1}(\mathbb{N}),\; \sum\limits_{n\in\mathbb{N}}E_n<+\infty.  
\end{align}
Applying Lemma \ref{lem 2.2, Combettes&Vu}, we have that $(D_n)_{n\in\mathbb{N}}$ converges.
This means that $(D_n)_{n\in\mathbb{N}}$ bounded and therefore $(\|x_{n+1}-x\|_{U_{n+1}^{-1}}^2)_{n\in\mathbb{N}}$ and $(M_n \|p_{n-1}-\tilde{p}_n\|^2)_{n\in\mathbb{N}}$ are bounded. Since $M_n=\frac{1}{2\mu}-\gamma_n^2\beta^2\mu\geq \frac{1}{2\mu}-\lambda^2\beta^2\mu $ and $\lambda < \frac{1}{\sqrt{10}\mu\beta} < \frac{1}{\sqrt{2}\beta\mu}$, then $\liminf\limits_{n\rightarrow\infty} M_n> 0$. Therefore we also have that $(\|p_{n-1}-\tilde{p}_n\|^2)_{n\in\mathbb{N}}$ is bounded. Consequently, there are $\theta$ and $\zeta$ in $\mathbb{R}$ such that $\theta=\sup\limits_{n\in\mathbb{N}}\|x_n-x\|^2_{U^{-1}_n}$ and $\zeta=\sup\limits_{n\in\mathbb{N}}\|p_{n-1}-\tilde{p}_n\|^2$, respectively. Now, consider (\ref{D_n}) again that $(\forall n\in\mathbb{N})$
\begin{align}\label{Before Telescope 4}
\|x_{n+1}-x\|_{U_{n+1}^{-1}}^2 + M_n\|p_{n-1}-\tilde{p}_n\|^2 &\leq  \left( 1+ \tilde{\eta}_n \right)  \left[ \|x_n- x\|^2_{U_n^{-1}} +4\gamma_{n-1}^{2}\mu\beta^2\|p_{n-2}-\tilde{p}_{n-1}\|^2  \right ] +E_n
\notag\\
&= \|x_n- x\|^2_{U_n^{-1}} +4\gamma_{n-1}^{2}\mu\beta^2\|p_{n-2}-\tilde{p}_{n-1}\|^2
\notag\\
&+ \tilde{\eta}_n \|x_n- x\|^2_{U_n^{-1}} +\tilde{\eta}_n 4\gamma_{n-1}^{2}\mu\beta^2\|p_{n-2}-\tilde{p}_{n-1}\|^2   +E_n
\notag\\
&\leq \|x_n- x\|^2_{U_n^{-1}} +4\gamma_{n-1}^{2}\mu\beta^2\|p_{n-2}-\tilde{p}_{n-1}\|^2
\notag\\
&+ \tilde{\eta}_n \theta +\tilde{\eta}_n 4\gamma_{n-1}^{2}\mu\beta^2\zeta   +E_n.
\end{align}
For convenience, we let $\tilde{M}_{n-1} = 4\gamma_{n-1}^{2}\mu\beta^2$ and we will show that $\liminf\limits_{n\rightarrow +\infty}( M_n-\tilde{M}_{n}) >0$. Since $(\forall n\in\mathbb{N}) \; \gamma_n\leq\lambda <\frac{1}{\sqrt{10}\mu\beta}$ which implies that $4\gamma_n^2\mu\beta^2+\gamma_n^2\mu\beta^2\leq 4\lambda^2\mu\beta^2+\lambda^2\mu\beta^2<\frac{1}{2\mu}$. Hence $\limsup\limits_{n\rightarrow +\infty}\left(4\gamma_n^2\mu\beta^2+\gamma_n^2\mu\beta^2\right) < \frac{1}{2\mu}$ and so $\liminf\limits_{n\rightarrow + \infty} \left[ \frac{1}{2\mu} -\left( 4\gamma_n^2\mu\beta^2+\gamma_n^2\mu\beta^2 \right)\right] >0$, this means that $\frac{1}{2\mu}- \left( 4\gamma_n^2\mu\beta^2 +\gamma_n^2\mu\beta^2 \right) > \epsilon > 0$ for some $\epsilon\in\mathbb{R}$ or equivalently to $M_n-\tilde{M}_n = \left[\frac{1}{2\mu} - \gamma_n^2\mu\beta^2\right]-\left[  4\gamma_n^2\mu\beta^2 \right]= \frac{1}{2\mu}- \left( 4\gamma_n^2\mu\beta^2+\gamma_n^2\mu\beta^2 \right) > \epsilon > 0$ for some $\epsilon\in\mathbb{R}$. Then it follows from (\ref{Before Telescope 4}) that
\begin{align}\label{For Telescoping}
\|x_{n+1}-x\|_{U_{n+1}^{-1}}^2 + \left(M_n-\tilde{M}_n \right)\|p_{n-1}-\tilde{p}_n\|^2 + \tilde{M}_n \|p_{n-1}-\tilde{p}_n\|^2&\leq  \|x_n- x\|^2_{U_n^{-1}} 
\notag\\
&+4\gamma_{n-1}^{2}\mu\beta^2\|p_{n-2}-\tilde{p}_{n-1}\|^2
\notag\\
&+ \tilde{\eta}_n \theta +\tilde{\eta}_n 4\gamma_{n-1}^{2}\mu\beta^2\zeta   +E_n
\notag\\
&= \|x_n- x\|^2_{U_n^{-1}} 
\notag\\
&+\tilde{M}_{n-1}\|p_{n-2}-\tilde{p}_{n-1}\|^2
\notag\\
&+ \tilde{\eta}_n \theta +\tilde{\eta}_n \tilde{M}_{n-1}\zeta   +E_n.
 \end{align}

\noindent We apply Lemma \ref{lem 3.1 (7)}(iii) with the setting of $\chi =1$, 
$\alpha_n=  \|x_n- x\|^2_{U_n^{-1}} +\tilde{M}_{n-1}\|p_{n-2}-\tilde{p}_{n-1}\|^2 $ , $\beta_n= \left(M_n-\tilde{M}_n \right)\|p_{n-1}-\tilde{p}_n\|^2$, $\epsilon_n=  \tilde{\eta}_n \theta +\tilde{\eta}_n \tilde{M}_{n-1}\zeta   +E_n$. It follows from (\ref{tilde(eta_n) and sum of En bounded}), the fact that  $(\forall n\in\mathbb{N}),\; M_n, \tilde{M}_n$ are bounded $\left( \mbox{since } M_n = \frac{1}{2\mu}-\gamma^2_n\beta^2\mu, \tilde{M}_n = 4\gamma_n^2\mu(1+\beta^2)\right)$, and $ M_n-\tilde{M}_n > \epsilon>0$ for some $\epsilon\in\mathbb{R}$, that is $\sum\limits_{n\in\mathbb{N}} \left(M_n-\tilde{M}_n \right)\|p_{n-1} \!-\! \tilde{p}_n\|^2 < +\infty$ and so
\begin{align}\label{sum of |p_(n-1)-tilda(p)n|^2 is bounded}
\sum\limits_{n\in\mathbb{N}} \|p_{n-1} \!-\! \tilde{p}_n\|^2 < +\infty.
\end{align}

\noindent It follows from (\ref{ConditionA}), (\ref{|tilde(x)_n+1 - x|^2_U_n(-1)}), Lemma \ref{lem2.1(10)}, (\ref{sum of |p_(n-1)-tilda(p)n|^2 is bounded}), $(\gamma_n)_{n\in\mathbb{N}}\leq \lambda$, and $(\eta_n)$, $(x_n)_{n\in\mathbb{N}}$ are bounded that
\begin{small}
\begin{align}\label{B1}
(\forall n\in\mathbb{N})\; \|z_n\|^2_{U_{n+1}^{-1}}=\|\tilde{x}_{n+1}-x\|^2_{U_{n+1}^{-1}}&\leq(1+\eta_n)\|\tilde{x}_{n+1}-x\|^2_{U_{n}^{-1}}\notag\\
&\leq (1+\eta_n)\left[ \|x_n-x\|^2_{U_n^{-1}}-\mu^{-1}\|x_n-\tilde{p}_n\|^2 +\gamma^2_n\beta^2\mu\|p_{n-1}-\tilde{p}_n\|^2  \right] \notag\\
&\leq (1+\eta_n) \|x_n-x\|^2_{U_n^{-1}}+(1+\eta_n)\gamma^2_n\beta^2\mu\|p_{n-1}-\tilde{p}_n\|^2. \notag\\
&=  \|x_n-x\|^2_{U_n^{-1}} + \eta_{n} \|x_n-x\|^2_{U_n^{-1}}+ \gamma^2_n\beta\mu\|p_{n-1}-\tilde{p}_n\|^2 + \eta_n \gamma^2_n\beta\mu\|p_{n-1}-\tilde{p}_n\|^2,
\end{align}
\end{small}
hence $\sup\limits_{n\in\mathbb{N}}\|z_n\|^2_{U_{n+1}^{-1}}<+\infty$.
It follows from $x_{n+1} - x = z_n - e_n$ and  (\ref{B1}) that
\begin{align}\label{quasi-Fejer monotone}
(\forall n\in \mathbb{N})\; \|x_{n+1}-x\|_{U_{n+1}^{-1}}^2 &= \|z_n\|_{U_{n+1}^{-1}}^2 -2\inner{z_n}{e_n}_{U_{n+1}^{-1}} + \|e_n\|_{U_{n+1}^{-1}}^2
\notag\\
&\leq \left[ (1+\eta_n)\left( \|x_n-x\|^2_{U_n^{-1}}-\mu^{-1}\|x_n-\tilde{p}_n\|^2 +\gamma^2_n\beta^2\mu\|p_{n-1}-\tilde{p}_n\|^2  \right)  \right] 
\notag\\
&+ 2\|z_n\|_{U_{n+1}^{-1}} \|e_n\|_{U_{n+1}^{-1}} + \|e_n\|_{U_{n+1}^{-1}}^2
\notag\\
&\leq (1+\eta_n) \|x_n-x\|^2_{U_n^{-1}} - (1+\eta_n) \mu^{-1}\|x_n-\tilde{p}_n\|^2 + \tilde{E}_n
\notag\\
&\leq  (1+\eta_n) \|x_n-x\|^2_{U_n^{-1}} + \tilde{E}_n,
\end{align}
where $\tilde{E}_n = \gamma^2_n\beta^2\mu\|p_{n-1}-\tilde{p}_n\|^2  + \eta_n \gamma^2_n\beta^2\mu\|p_{n-1}-\tilde{p}_n\|^2 + 2\|z_n\|_{U_{n+1}^{-1}} \|e_n\|_{U_{n+1}^{-1}} + \|e_n\|_{U_{n+1}^{-1}}^2$, in which $\sum\limits_{n\in\mathbb{N}} \tilde{E}_n < +\infty$, because (\ref{sum|e_n| bounded}), (\ref{sum of |p_(n-1)-tilda(p)n|^2 is bounded}), (\ref{B1}), $\beta\in (0,+\infty)$, $(\gamma_n)_{n\in\mathbb{N}}< \lambda$ and $\eta_n\in\ell_{+}^1(\mathbb{N})$.

The inequality (\ref{quasi-Fejer monotone}) shows that $(x_n)_{n\in\mathbb{N}}$ is $|\cdot|^2$- quasi-Fejer monotone with respect to the target set $zer(A+B)$ relative to $(U_{n}^{-1})_{n\in\mathbb{N}}$. Moreover, by Proposition \ref{Prop 3.2 (10)}, $(\|x_n-x\|_{U_n^{-1}})_{n\in\mathbb{N}}$ is bounded.

\noindent It follows from (\ref{|x_n-p_(n-1)|}) and (\ref{sum of |p_(n-1)-tilda(p)n|^2 is bounded}) that 
\begin{align*}
\|x_n - \tilde{p}_n \| ^2 &\leq  2\|x_n - p_{n-1}\|^2 + 2 \| p_{n-1}-\tilde{p}_n\|^2 \notag\\
&\leq 2\left[\gamma_{n-1} \mu \beta \|p_{n-2}-p_{n-1} \|+\gamma_{n-1} \mu \left( \|a_{n-1}\| + \|c_{n-1}\| \right)\right]^2 + 2 \| p_{n-1}-\tilde{p}_n\|^2\\
&\leq 2\left( 2 \left( \gamma^2_{n-1} \mu^2 \beta^2 \left( \|p_{n-2} - \tilde{p}_{n-1}\| + \|\tilde{p}_{n-1}-p_{n-1}\| \right)^2  +\gamma^2_{n-1} \mu^2 \left( \|a_{n-1}\| + \|c_{n-1}\| \right)^2\right)\right) \\
&+ 2 \| p_{n-1}-\tilde{p}_n\|^2\\
&\leq 4\left( 2\gamma^2_{n-1} \mu^2 \beta^2 \left( \|p_{n-2} - \tilde{p}_{n-1}\|^2 + \|\tilde{p}_{n-1}-p_{n-1}\|^2 \right) \right)  +2  \gamma^2_{n-1} \mu^2  2\left( \|a_{n-1}\|^2 + \|c_{n-1}\|^2 \right)\\
&+ 2 \| p_{n-1}-\tilde{p}_n\|^2\\
&\leq 8 \gamma^2_{n-1} \mu^2 \beta^2 \left( \|p_{n-2} - \tilde{p}_{n-1}\|^2 + \|\tilde{p}_{n-1}-p_{n-1}\|^2 \right) + 4\gamma^2_{n-1} \mu^2  \left( \|a_{n-1}\|^2 + \|c_{n-1}\|^2 \right) \\
&+ 2 \| p_{n-1}-\tilde{p}_n\|^2
\end{align*}
and therefore (see (\ref{|yn-yilde(y)n|,|pn-tilde(p)n|,|qn-tilde(q)n|}))
\begin{equation}\label{sum |xn- tilde(p)_n| bounded_edit}
\sum\limits_{n\in\mathbb{N}}\|x_n-\tilde{p}_n\|^2 < +\infty.
\end{equation}

(i) It follows from (\ref{sum |xn- tilde(p)_n| bounded_edit}) and (\ref{the sum of |yn-yilde(y)n|,|pn-tilde(p)n|,|qn-tilde(q)n| bounded}) that
\begin{align}
\sum\limits_{n\in\mathbb{N}} \|x_{n} - {p}_n\|^2 \leq 2 \sum\limits_{n\in\mathbb{N}} \|x_{n} - \tilde{p}_n\|^2 + 2 \sum\limits_{n\in\mathbb{N}} \|p_{n} - \tilde{p}_n\|^2 < +\infty
\end{align}
Futhermore, we can derive form  (\ref{set_para2}), (\ref{the sum of |yn-yilde(y)n|,|pn-tilde(p)n|,|qn-tilde(q)n| bounded}), (\ref{sum|e_n| bounded}), (\ref{sum of |p_(n-1)-tilda(p)n|^2 is bounded}) and (\ref{sum |xn- tilde(p)_n| bounded_edit}) that
\begin{align}
\sum\limits_{n\in\mathbb{N}} \|y_{n} - {q}_n\|^2 &\leq \sum\limits_{n\in\mathbb{N}} \|\tilde{q}_n - \tilde{y}_n + q_{n} -\tilde{q}_n + \tilde{y}_n - {y}_n\|^2 \notag\\
&=\sum\limits_{n\in\mathbb{N}} \|\tilde{q}_n - \tilde{y}_n - e_n\|^2 \notag\\
&= \sum\limits_{n\in\mathbb{N}} \|\tilde{p}_n -\gamma_nU_nB(\tilde{p}_n) - \left(x_n-\gamma_nU_nB(p_{n-1})\right) - e_n\|^2\notag\\
&= \sum\limits_{n\in\mathbb{N}} \|\tilde{p}_n -x_n + \gamma_nU_n \left(B(p_{n-1}) - B(\tilde{p}_{n})\right) - e_n\|^2\notag\\
&\leq 3 \sum\limits_{n\in\mathbb{N}} \left(   \|\tilde{p}_n -x_n\|^2 + \gamma_n^2 \|U_n\|^2 \|B(p_{n-1}) - B(\tilde{p}_{n})\|^2 + \|e_n\|^2 \right)\notag\\
&\leq 3 \sum\limits_{n\in\mathbb{N}} \left(   \|\tilde{p}_n -x_n\|^2 + \gamma_n^2 \mu^2 \beta^2\|p_{n-1} - \tilde{p}_{n}\|^2 + \|e_n\|^2 \right)\notag\\
&< +\infty
\end{align}

(ii) We want to show that $x_n\rightharpoonup \bar{x}$ and $p_n\rightharpoonup \bar{x}$ for some $\bar{x}\in zer(A+B)$. Let $x$ be a weak cluster point of $(x_n)_{n\in\mathbb{N}}$. Then there exists a subsequence $(x_{k_n})_{n\in\mathbb{N}}$ that converges weakly to $x$. By (\ref{sum |xn- tilde(p)_n| bounded_edit}), we know that $\sum\limits_{n\in\mathbb{N}} \|x_{n} - \tilde{p}_n\|^2 < +\infty$, then $\tilde{p}_{k_n}\rightharpoonup x$. Furthermore, it follows from (\ref{set_para2}), (\ref{u_n}), (\ref{sum of |p_(n-1)-tilda(p)n|^2 is bounded}), (\ref{sum |xn- tilde(p)_n| bounded_edit}) and $\liminf\limits_{n\rightarrow\infty} \gamma_n > 0$ that $u_{k_n} =\gamma_{k_n}^{-1} U_{k_n}^{-1}(x_{k_n}-\tilde{p}_{k_n})+B(\tilde{p}_{k_n})-B(p_{k_{n-1}}) \rightarrow 0$. By (\ref{u_n}) we also know that $(\tilde{p}_{k_n}, u_{k_n})\in Gra(A+B)$. By Proposition \ref{Prop 20.33 (BC)}, we obtain that $(x,0)\in Gra (A+B)$ and then $x\in zer(A+B)$. Altogether, it follows from (\ref{quasi-Fejer monotone}), Lemma \ref{lem 2.3 (10)} and Theorem \ref{Thm 3.3 (10)} that $x_n\rightharpoonup \bar{x}$ and hence $p_n \rightharpoonup \bar{x}$, by using (i)$\left[\sum\limits_{n\in\mathbb{N}} \|x_{n} - {p}_n\|^2 < +\infty\right] $.
\end{proof}


\section{A Primal-Dual Solver for Monotone Inclusion Problem}

As we know, many non-smooth optimization problems can be written as monotone inclusion primal-dual problems. In this case, we want to enhance our algorithm to deal with the \textbf{Problem 1}. So we proposed a corollary which follows from Theorem \ref{Thm_TengEP_variable_Error} as below.

\begin{corollary}\label{Mono_Inclusion_Corollary}
Let $\alpha$ be in $(0,+\infty)$, let $(\eta_{0, n})_{n\in\mathbb{N}}$ be a sequence in $\ell^{1}_{+} (\mathbb{N})$, let $(U_n)_{n\in\mathbb{N}}$ be a sequence in $\mathcal{P}_{\alpha}(\mathcal{H})$, and for every $i\in\{ 1,\dots, m\}$, let $(\eta_{i, n})_{n\in\mathbb{N}}$ be a sequence in $\ell_+^1 (\mathbb{N})$, let $(U_{i,n})_{n\in\mathbb{N}}$ be a sequence in $\mathcal{P}_{\alpha} (\mathcal{G}_i)$ such that $\mu = sup_{n\in\mathbb{N}} \{\|U_n\|, \|U_{1,n} \|,\dots, \|U_{m,n}\|  \}< +\infty$ and 
\begin{align}
(\forall n\in\mathbb{N}) \quad (1+\eta_{0,n}) U_{n+1} &\succcurlyeq U_n \notag\\
\mbox{and} \; (\forall i\in \{1,\dots,m\})\quad (1+\eta_{i,n}) U_{i,n+1}&\succcurlyeq U_{i,n}.
\end{align}
Let $(a_{1,n})_{n\in\mathbb{N}}, (b_{1,n})_{n\in\mathbb{N}}$ and $(c_{1,n})_{n\in\mathbb{N}}$ be absolutely summable sequences in $\mathcal{H}$, and for every $i\in\{1,\dots,m\}$, let $(a_{2,i,n})_{n\in\mathbb{N}}, (b_{2,i,n})_{n\in\mathbb{N}}$ and $(c_{2,i,n})_{n\in\mathbb{N}}$ be absolutely summable sequences in $\mathcal{G}_i$. Furthermore, set 
\begin{align}
\beta = v_0 + \sqrt{\sum\limits_{i=1}^m \|L_i\|^2},
\end{align}
let $x_0\in\mathcal{H}$, let $(v_{1,0},\dots, v_{m,0})\in\mathcal{G}_1 \bigoplus \dots \bigoplus \mathcal{G}_m$, let $(\gamma_n)_{n\in\mathbb{N}}\leq\lambda$ with $\lambda < \frac{1}{\sqrt{10}\mu\beta}$ and $\liminf\limits_{n\rightarrow +\infty} \gamma_n > 0$.
Set
\begin{align}\label{Mono_inclusion_Algor}
(\forall n\in\mathbb{N}) \left\lfloor
\begin{array}{l}
y_{1,n} = x_n - \gamma_n U_n \left( C(p_{1,n-1}) + \sum\limits_{i=1}^m L^*_i (p_{2,i,n-1}) + a_{1,n} \right)  \\
\mbox{for}\; i = 1,\dots,m \\
\left\lfloor
\begin{array}{l}
y_{2,i,n} = v_{i,n} + \gamma_n U_{i,n} \left( L_i(p_{1,n-1}) + a_{2,i,n}  \right)\\
p_{2,i,n} = J_{\gamma_{n} U_{i,n} B^{-1}_{i}}(y_{2,i,n} - \gamma_n U_{i,n} r_i) + b_{2,i,n} \\
\end{array}
\right.\\
p_{1,n} = J_{\gamma_n U_n A} (y_{1,n}+ \gamma_n U_n z) + b_{1,n} \\
\mbox{for}\; i = 1,\dots,m \\
\left\lfloor
\begin{array}{l}
q_{2,i,n} = p_{2,i,n} + \gamma_n U_{i,n} \left( L_i (p_{1,n}) + c_{2,i,n} \right)\\
v_{i,n+1} = v_{i,n} - y_{2,i,n} + q_{2,i,n} \\
\end{array}
\right.\\
q_{1,n} = p_{1,n} - \gamma_n U_{n} \left( C(p_{1,n}) + \sum\limits_{i=1}^{m} L_i^* (p_{2,i,n}) + c_{1,n}  \right) \\
x_{n+1} = x_n - y_{1,n} + q_{1,n}\\
\end{array}
\right.
\end{align}
Then the following hold.
\begin{enumerate}[label=(\roman*)]
\item  $\sum\limits_{n\in\mathbb{N}} \|x_n-p_{1,n}\|^2 < +\infty$ and $(\forall i\in\{1,\dots,m\}) \sum\limits_{n\in\mathbb{N}}\|v_{i,n}-p_{2,i,n}\|^2 < +\infty$.
\item There exists a solution $\bar{x}$ to (\ref{Mono Inclusion 2}) and a solution $(\bar{v}_1,\dots,\bar{v}_m)$ to (\ref{Mono Inclusion 3}) such that the following hold. 
\begin{enumerate}
\item[(1)]  $x_n\rightharpoonup \bar{x}$ and $p_{1,n}\rightharpoonup\bar{x}$.
\item[(2)] $(\forall i\in\{1,\dots,m\}\;) v_{i,n}\rightharpoonup \bar{v}_i$ and $p_{2,i,n}\rightharpoonup \bar{v}_i$.

\end{enumerate}
\end{enumerate}
\end{corollary}


\begin{proof}
All sequences generated by algorithm (\ref{Mono_inclusion_Algor}) are well defined by Lemma \ref{lem3.7(11)}. We define $\mathbfcal{H} = \mathcal{H}\bigoplus \mathcal{G}_1 \bigoplus \dots \bigoplus \mathcal{G}_m$. the Hilbert direct sum of the Hilbert space $\mathcal{H}$ and $(G_i)_{1\leq i\leq m}$, the scalar product and the associated norm of $\mathbfcal{H}$ respectively  defined by
\begin{align}\label{Proof_mono_inclu_1}
\langle\langle\langle \cdot \rangle\rangle\rangle &: ((x,v),(y,w)) \mapsto \inner{x}{y} + \sum\limits_{i=1}^m \inner{v_i}{w_i}, and \notag\\
\|\| \cdot \|\| &: (x,v) \mapsto \sqrt{\|x\|^2 + \sum_{i=1}^m \|v_i\|^2},
\end{align}
where $\mathbf{v}= (v_1,\dots,v_m)$ and $\mathbf{w} = (w_1,\dots,w_m)$ are generic elements in $\mathcal{G}_1 \bigoplus \dots \bigoplus \mathcal{G}_m$. Set 
\begin{align}\label{Proof_mono_inclu_2}
\begin{cases}
\mathbf{A} : \mathbfcal{H} \rightarrow 2^{\mathbfcal{H}} : (x,v_1,\dots,v_m) \mapsto (-z + Ax)\times(r_1+B^{-1}_1 v_1)\times\dots\times(r_m +B^{-1}_m v_m)\\
\mathbf{B}: \mathbfcal{H} \rightarrow \mathbfcal{H}: (x,v_1,\dots,v_m) \mapsto (Cx + \sum\limits_{i=1}^m L_i^* v_i, -L_1x,\dots,-L_m x)\\
(\forall n\in\mathbb{N})\;  \mathbf{U}_n : \mathbfcal{H} \rightarrow \mathbfcal{H}: (x,v_1,\dots,v_m) \mapsto (U_nx, U_{1,n}v_1,\dots,U_{m,n}v_m) 
\end{cases}
\end{align}
Since $\mathbf{A}$ is maximally monotone (see  Proposition 20.22 and 20.23 in [\cite{BC-Book}]), $\mathbf{B}$ is monotone $\beta$-Lipschitzian (see Equation (3.10) in [\cite{Combettes2012Pesquet}]) with $dom \mathbf{B} =\mathbfcal{H}$, $\mathbf{A}+\mathbf{B}$ is maximally monotone (see  Corolarry 24.24(i) in [\cite{Combettes2012Pesquet}]). Now set $(\forall n\in\mathbb{N})\; \eta_n = \max\{\eta_{0,n},\eta_{1,n},\dots,\eta_{m,n}\}$. Then $(\eta_n)_{n\in\mathbb{N}}\in \ell^1_{+}(\mathbb{N}).$ Moreover, we derive from our assumptions on the sequences $(U_n)_{n\in\mathbb{N}}$ and $(U_{1,n})_{n\in\mathbb{N}},\dots,(U_{m,n})_{n\in\mathbb{N}}$ that
\begin{align}\label{Proof_mono_inclu_3}
\mu =\sup_{n\in\mathbb{N}} \|\mathbf{U}_n\| < +\infty \quad\mbox{and}\quad (\forall n\in\mathbb{N})\; (1+\eta_n)\mathbf{U}_{n+1}\succcurlyeq \mathbf{U}_n \in \mathcal{P}_{\alpha}(\mathbfcal{H}).
\end{align}
In addition, Proposition 23.15(ii) and 23.16 in [\cite{BC-Book}] yields $(\forall \gamma\in (0,+\infty) (\forall n\in\mathbb{N})(\forall(x,v_1,\dots,v_m)\in \mathbfcal{H})$
\begin{align}\label{Proof_mono_inclu_4}
J_{\gamma U_n \mathbf{A}}(x,v_1,\dots,v_m)= \left(
J_{\gamma U_n A}(x+\gamma U_n z),(J_{\gamma U_{i,n }B^{-1}_i } (v_i-\gamma U_{i,n}r_i) )_{1\leq i \leq m}
 \right).
\end{align}
It is shown in Equation (3.12) and Equation (3.13) of [\cite{Combettes2012Pesquet}] that under the condition (\ref{Mono_Inclusion 1}), $zer(\mathbf{A}+\mathbf{B})\neq\emptyset.$ Moreover, Equation (3.21) ans Equation (3.22) in [\cite{Combettes2012Pesquet}] yield
\begin{align}\label{Proof_mono_inclu_5}
(\bar{x},\bar{v}_1,\dots,\bar{v}_m)\in zer(\mathbf{A}+\mathbf{B}) \Rightarrow \bar{x} \;\mbox{solves (\ref{Mono Inclusion 2}) and } (\bar{v}_1,\dots,\bar{v}_m) \mbox{ solves (\ref{Mono Inclusion 3}). }
\end{align}
Let us next set

\begin{align}\label{Proof_mono_inclu_6}
(\forall n\in\mathbb{N})
\begin{cases} 
\mathbf{x}_n = (x_n, v_{1,n},\dots,v_{m,n})\\
\mathbf{y}_n = (y_{1,n},y_{2,1,n},\dots,y_{2,m,n})\\
\mathbf{p}_n = (p_{1,n},p_{2,1,n},\dots,p_{2,m,n})\\
\mathbf{q}_n = (q_{1,n}, q_{2,1,n},\dots,q_{2,m,n}) 
\end{cases} \mbox{and} 
\begin{cases}
\mathbf{a}_n = (a_{1,n},a_{2,1,n},\dots,a_{2,m,n})\\
\mathbf{b}_n = (b_{1,n},b_{2,1,n},\dots,b_{2,m,n})\\
\mathbf{c}_n = (c_{1,n},c_{2,1,n},\dots,c_{2,m,n}).
\end{cases}
\end{align}

Then our assumptions imply that
\begin{align}\label{Proof_mono_inclu_7}
\sum\limits_{n\in\mathbb{N}} \|\| \mathbf{a}_n\|\| < \infty, \sum\limits_{n\in\mathbb{N}} \|\| \mathbf{b}_n\|\| < \infty, \mbox{and} \sum\limits_{n\in\mathbb{N}} \|\| \mathbf{c}_n\|\| < \infty.
\end{align}
Furthermore, it follows from the definition of $\mathbf{B}$, (\ref{Proof_mono_inclu_4}), and (\ref{Proof_mono_inclu_6}) that (\ref{Mono_inclusion_Algor}) can be written in $\mathbfcal{H}$ as 
\begin{align}
\left\lfloor
\begin{array}{l}
\mathbf{y_n} = \mathbf{x}_n -\gamma_n \mathbf{U}_n(B(\mathbf{p}_{n-1})+ \mathbf{a}_n)\\
\mathbf{p}_n = J_{\gamma_n U_n A} \mathbf{y}_n + \mathbf{b}_n\\
\mathbf{q}_n = \mathbf{p}_n - \gamma_n U_n(B(\mathbf{p}_n)+\mathbf{c}_n)\\
\mathbf{x}_{n+1}=\mathbf{x}_n - \mathbf{y}_n +\mathbf{q}_n,
\end{array}
\right.
\end{align}
which is (\ref{set_para1}). Moreover, every specific conditions in Theorem \ref{Thm_TengEP_variable_Error} are satisfied.
\begin{enumerate}[label=(\roman*)]
\item By Theorem \ref{Thm_TengEP_variable_Error}(i), $\sum\limits_{n\in\mathbb{N}} \| \| x_n - p_n\|\|^2 < +\infty$.
\item  There exists a solution $\bar{x}$ to (\ref{Mono Inclusion 2}) and a solution $(\bar{v}_1,\dots,\bar{v}_m)$ to (\ref{Mono Inclusion 3}) such that the following hold.
\begin{enumerate}
\item[(a)] $x_n\rightharpoonup \bar{x}$ and $p_{1,n}\rightharpoonup \bar{x}$.
\item[(b)] $(\forall i\in\{ 1,\dots,m \})$ $v_{i,n}\rightharpoonup \bar{v}_i$ and $p_{2,i,n}\rightharpoonup \bar{v}_i$. 
\end{enumerate}

\end{enumerate}
\end{proof}


\section{A Primal-Dual Splitting Algorithm for Convex Optimization Problem}

Next, we further introduce the primal-dual splitting algorithm for solving \textbf{Problem 2}. Actually, we can call it splitting algorithm because the involved functions in our problem are decoupled, as we can see in the structure of the algorithm below.

\begin{theorem}\label{Primal-Dual Tseng-EP with var matrices and error}
 In \textbf{Problem 2}, suppose that
 \begin{align}\label{Primal-Dual 3}
 z\in ran \left( \partial f + \sum\limits_{i=1}^m L_i^*(\partial g_i)(L_i \cdot - r_i) +\nabla h \right).
 \end{align}
Let $\alpha$ be in $(0,+\infty)$, let $(\eta_{0, n})_{n\in\mathbb{N}}$ be a sequence in $\ell^{1}_{+} (\mathbb{N})$, let $(U_n)_{n\in\mathbb{N}}$ be a sequence in $\mathcal{P}_{\alpha}(\mathcal{H})$, and for every $i\in\{ 1,\dots, m\}$, let $(\eta_{i, n})_{n\in\mathbb{N}}$ be a sequence in $\ell_+^1 (\mathbb{N})$, let $(U_{i,n})_{n\in\mathbb{N}}$ be a sequence in $\mathcal{P}_{\alpha} (\mathcal{G}_i)$ such that $\mu = sup_{n\in\mathbb{N}} \{\|U_n\|, \|U_{1,n} \|,\dots, \|U_{m,n}\|  \}< +\infty$ and 
\begin{align}
(\forall n\in\mathbb{N}) \quad (1+\eta_{0,n}) U_{n+1} &\succcurlyeq U_n \notag\\
\mbox{and} \; (\forall i\in \{1,\dots,m\})\quad (1+\eta_{i,n}) U_{i,n+1}&\succcurlyeq U_{i,n}.
\end{align} 
 
\noindent Let $(a_{1,n})_{n\in\mathbb{N}}, (b_{1,n})_{n\in\mathbb{N}}$ and $(c_{1,n})_{n\in\mathbb{N}}$ be absolutely summable sequences in $\mathcal{H}$, and for every $i\in\{1,\dots,m\}$, let $(a_{1,i,n})_{n\in\mathbb{N}}, (b_{1,i,n})_{n\in\mathbb{N}}$ and $(c_{1,i,n})_{n\in\mathbb{N}}$ be absolutely summable sequences in $\mathcal{G}_i$. Furthermore, set 
\begin{align}\label{Primal-Dual 4}
\beta = v_0 + \sqrt{\sum\limits_{i=1}^m \|L_i\|^2},
\end{align}
let $x_0\in\mathcal{H}$, let $(v_{1,0},\dots, v_{m,0})\in\mathcal{G}_1 \bigoplus \dots \bigoplus \mathcal{G}_m$, let $(\gamma_n)_{n\in\mathbb{N}}\leq\lambda$ with $\lambda < \frac{1}{\sqrt{10}\mu\beta}$ and $\liminf\limits_{n\rightarrow +\infty} \gamma_n > 0$.
Set
\begin{align}\label{Primal-Dual 5_Algor}
(\forall n\in\mathbb{N}) \left\lfloor
\begin{array}{l}
y_{1,n} = x_n - \gamma_n U_n \left(\nabla h(p_{1,n-1}) + \sum\limits_{i=1}^m L^*_i (p_{2,i,n-1}) + a_{1,n} \right)  \\
\mbox{for}\; i = 1,\dots,m \\
\left\lfloor
\begin{array}{l}
y_{2,i,n} = v_{i,n} + \gamma_n U_{i,n} \left( L_i(p_{1,n-1}) + a_{2,i,n}  \right)\\
p_{2,i,n} = prox_{\gamma_{n} g_{i}^* }^{U_{i,n}^{-1}} (y_{2,i,n} - \gamma_n U_{i,n} r_i) + b_{2,i,n} \\
\end{array}
\right.\\
p_{1,n} = prox_{\gamma_n f}^{U_n^{-1}} (y_{1,n}+ \gamma_n U_n z) + b_{1,n} \\
\mbox{for}\; i = 1,\dots,m \\
\left\lfloor
\begin{array}{l}
q_{2,i,n} = p_{2,i,n} + \gamma_n U_{i,n} \left( L_i (p_{1,n}) + c_{2,i,n} \right)\\
v_{i,n+1} = v_{i,n} - y_{2,i,n} + q_{2,i,n} \\
\end{array}
\right.\\
q_{1,n} = p_{1,n} - \gamma_n U_{n} \left( \nabla h (p_{1,n}) + \sum\limits_{i=1}^{m} L_i^* (p_{2,i,n}) + c_{1,n}  \right) \\
x_{n+1} = x_n - y_{1,n} + q_{1,n}\\
\end{array}
\right.
\end{align}
Then the following hold.
\begin{enumerate}[label=(\roman*)]
\item  $\sum\limits_{n\in\mathbb{N}} \|x_n-p_{1,n}\|^2 < +\infty$ and $(\forall i\in\{1,\dots,m\}) \sum\limits_{n\in\mathbb{N}}\|v_{i,n}-p_{2,i,n}\|^2 < +\infty$.
\item There exists a solution $\bar{x}$ to (\ref{Primal-Dual 1}) and a solution $(\bar{v}_1,\dots,\bar{v}_m)$ to (\ref{Primal-Dual 2}) such that the following hold.
\begin{enumerate}
\item[(a)] $z - \sum\limits_{j=1}^m L_j^*  \bar{v}_j \in \partial f(\bar{x}) + \nabla h(\bar{x}) $ and $(\forall i \in \{1,\dots,m \})\; L_i \bar{x} - r_i \in \partial g^*_i (\bar{v}_i)$.\\
\item[(b)] $x_n \rightharpoonup \bar{x}$ and $p_{1,n}\rightharpoonup \bar{x}$.\\
\item[(c)] $(\forall i\in\{1,\dots,m \}) \; v_{i,n}\rightharpoonup \bar{v}_i$ and $p_{2,i,n}\rightharpoonup \bar{v}_{i} $.
\end{enumerate}
\end{enumerate}
\end{theorem}

\begin{proof}
Let us define 
\begin{align}\label{Proof_Primal_Dual_Eq_1}
A = \partial f, \quad C=\nabla h \quad\mbox{and} \quad (\forall i = \{1,\dots,m\}) \; B_i = \partial g_i  
\end{align}
It clear that (\ref{Primal-Dual 3}) yields (\ref{Mono_Inclusion 1}) and using (\ref{(subdriff(f))inverse equals subdrif f_star}) and (\ref{prox^U_f and P^U_C}) that (\ref{Primal-Dual 5_Algor}) yields (\ref{Mono_inclusion_Algor}). Moreover, it follows from Theorem 20.40 in [\cite{BC-Book}] that the operators $A$ and $(B_i)_{1\leq i\leq m}$ are maximally monotone, and from  Proposition 17.10 in [\cite{BC-Book}] that $C$ is monotone which is a Lipschitzian operator by the hypothesis of Problem 2. Altogether, we can apply Corollary \ref{Mono_Inclusion_Corollary} to obtain the existence of a point $\bar{x}\in\mathcal{H}$ such that 
\begin{align}\label{Proof_Primal_Dual_Eq_2}
 z\in \partial f (\bar{x}) + \sum\limits_{i=1}^{m} L_i^* \left( \partial g_i (L_i \bar{x}- r_i)\right) + \partial h(\bar{x}),
\end{align}
and of an $m$-tuple $(\bar{v}_1,\dots,\bar{v}_m)\in \mathcal{G}_1\bigoplus\dots\bigoplus\mathcal{G}_m$ such that 
\begin{align}\label{Proof_Primal_Dual_Eq_3}
(\exists x\in\mathcal{H})\;
\begin{cases}
z-\sum\limits_{j=1}^m L_j^* \bar{v}_j \in \partial f(x) +\nabla h(x)\\
(\forall i \in \{1,\dots,m \})\; \bar{v}_i \in (\partial g_i) (L_i x - r_i),
\end{cases}
\end{align}
that satisfy $(i)$ and $(ii)$. Now we can follow the proof in [\cite{Combettes2012Pesquet}] with our setting above and some tools in [\cite{BC-Book}] to obtain that $\bar{x}$ solves (\ref{Primal-Dual 1}) and $(\bar{v}_1,\dots,\bar{v}_m)$ solves (\ref{Primal-Dual 2}).
\end{proof}


\begin{remark}
	In order to assure (\ref{Primal-Dual 3}), we need some similar conditions as given in [\cite{Combettes2012Pesquet}] (Proposition 4.3):
	Suppose that (\ref{Primal-Dual 1}) has at least one solution and set
	\begin{align}
	\mathbb{S}=\{ (L_ix-y_i)_{1\leq i \leq m} | x\in dom f and (\forall i\in \{1,\dots,m\}) y_i\in dom\; g_i \}.
	\end{align}
	Then the equation (\ref{Primal-Dual 3}) is satisfied if one of the following holds.
	\begin{enumerate}[label=(\roman*)]
	\item $(r_1,\dots,r_m)\in \textbf{sri}\; \mathbb{S}$. 
	\item For every $i\in\{1,\dots,m \}$, $g_i$ is real-valued.
	\item $\mathcal{H}$ and $(\mathcal{G}_i)_{1\leq i \leq m}$ are finite-dimentional, and there exists $x\in \textbf{ri}\; dom f$ such that
	\begin{align}
	(\forall i \in \{1,\dots,m\}) L_ix-r_i\in \textbf{ri}\; dom \;g_i.
	\end{align}
	The notations $\textbf{ri}$ and $\textbf{sri}$ denote to be a relative interior and strong relative interior of set respectively which we refer readers to see more detail in [\cite{BC-Book}].
	\end{enumerate}
\end{remark}


\section{Numerical Experiment in Imaging}

For this section, we intend to illustrate the numerical experiment in image deblurring which is correlated with our proposed primal-dual problem. Throughout this part, we implemented the numerical codes in MATLAB and performed all computations on a Window desktop with an Intel(R) Core(TM) i5-8250U processor at 1.6 gigahertz up to 1.8 gigahertz and RAM 8.00 GB. Accordingly, the theoretical result obtained in the previous section can be used. 
It should be noted that we use the grayscale image which have been normalized, in order to make their pixels range in the closed interval from 0 to 1 for this experiment. 

For a given matrix $A\in\mathbb{R}^{n\times n}$ describing a blur operator and a given vector $b\in\mathbb{R}^n$ representing the blurred and noisy image, the task is to estimate the unknown original image $\bar{x}\in\mathbb{R}^n$ fulfilling
\begin{align*}
	A\bar{x}=b.
\end{align*} 
To this end we solve the following regularized convex minimization problem
\begin{align}\label{TVisoDebluring probelm}
	\inf\limits_{x\in [0,1]^n} \big\lbrace \| Ax-b\|_1 +\lambda (TV_{iso}(x)+ \|x\|^2)   \big\rbrace,
\end{align}
where $\lambda > 0$ is a regularization parameter and $TV_{iso}: \mathbb{R}^n \rightarrow\mathbb{R}$ is the discrete isotropic total variation functional. In this context, $x\in \mathbb{R}^n$ represents the vectorized image  $X\in\mathbb{R}^{M\times N}$, where $n=M\cdot N$ and $x_{i,j}$ denotes the normalized value of the pixel located in the $i$th row and the $j$th column, for $i=1,\dots,M$ and $j = 1,\dots,N$.
The \textit{isotropic total variation} $TV_{iso} : \mathbb{R}^n \rightarrow \mathbb{R}$ is defined by
\begin{align*}
	TV_{iso}(x)= \sum\limits_{i=1}^{M-1} \sum\limits_{i=1}^{N-1} \sqrt{(x_{i+1,j} - x_{i,j})^2+(x_{i,j+1}-x_{i,j})^2} +\sum_{i=1}^{M-1} |x_{i+1,N}-x_{i,N} | + \sum_{j=1}^{N-1} |x_{M,j+1}-x_{M,j} |. 
\end{align*}

The optimization problem (\ref{TVisoDebluring probelm}) can be written in the framework of Problem (\ref{Primal-Dual 1}). We denote $\mathcal{Y}=\mathbb{R}^n\times\mathbb{R}^n$ and define the linear operator $\tilde{L}: \mathbb{R}^n\rightarrow\mathcal{Y}$, $x_{i,j} \mapsto (\tilde{L}_{1}x_{i,j}, \tilde{L}_{2}x_{i,j})$, where
\begin{align*}
	\tilde{L}_{1}x_{i,j} = 
				\begin{cases} x_{i+1,j}-x_{i,j}, \; &\mbox{if}\; i< M \\
							0, \;&\mbox{if}\; i=M
				\end{cases}
	and \; 
	\tilde{L}_{2} x_{i,j} = \begin{cases}
							x_{i,j+1} - x_{i,j}, \;&\mbox{if}\; j<N\\
							0, \;&\mbox{if}\; j=N
	\end{cases}.
\end{align*}
The operator $\tilde{L}$ represents a discretization of the gradient using reflexive (Neumann) boundary conditions and standard finite differences and fulfils $\|\tilde{L}\|^2\leq 8$. For the formula for its adjoint operator $\tilde{L}^*:\mathcal{Y}\rightarrow\mathbb{R}^n$, we refer to [\cite{Chambolle2004}].

	For $(y,z),(p,q)\in \mathcal{Y}$, we introduce the inner product
\begin{align*}
	\inner{(y,z)}{(p,q)} = \sum_{i=1}^{M}\sum_{j=1}^{N} y_{i,j}p_{i,j} + z_{i,j}q_{i,j}
\end{align*}
and define $\|(y,z)\|_{\times} = \sum_{i=1}^{M}\sum_{j=1}^{N} \sqrt{y_{i,j}^2+z_{i,j}^2}$. One can check that $\| \cdot\|_{\times}$ is a norm on $\mathcal{Y}$ and that for every $x\in\mathbb{R}^n$, it holds $TV_{iso}(x)=\|\tilde{L}x\|_{\times}$. The conjugate function $(\|\cdot\|_{\times})^*: \mathcal{Y}\rightarrow\bar{\mathbb{R}}$ of $\|\cdot\|_{\times}$ is for every $(p,q)\in \mathcal{Y}$ given by
\begin{align*}
	(\|\cdot\|)^*(p,q)=
	\begin{cases}
		0, \;&\mbox{if}\; \|(p,q)\|_{\times *} \leq 1\\
		+\infty, \;&\ otherwise
	\end{cases}
\end{align*} 
where
\begin{align*}
	\|(p,q)\|_{\times *} = \sup_{\|(y,z)\|_{\times}\leq 1} \inner{(p,q)}{(y,z)} = \max_{\substack{1\leq i\leq M\\ 1\leq j\leq N} } \sqrt{p_{i,j}^2+q_{i,j}^2}.
\end{align*}

Therefore, the optimization problem (\ref{TVisoDebluring probelm}) can be written in the form of
\begin{align*}
	\inf_{x\in\mathcal{H}} \big\lbrace f(x)+g_1(Ax)+ g_2(\tilde{L}x) + h(x) \big\rbrace,
\end{align*}
where $f: \mathbb{R}^n\rightarrow\bar{\mathbb{R}}, f(x)=\iota_{[0,1]^n}(x), g_1(y) = \|y-b\|_1, g_2:\mathcal{Y}\rightarrow\mathbb{R}, g_{2}(y,z)=\lambda\|(y,z)\|_{\times}$ and $h:\mathbb{R}^n\rightarrow\mathbb{R}, h(x)=\lambda\|x\|^2$ (notice that terms $r_i$ and $z$ are taken to be the zero vectors for $i=1,2$). For every $p\in\mathbb{R}^n$, it holds $g_1^*(p)=\iota_{[-1,1]^n}(p)+p^Tb$, while for every $(p,q)\in\mathcal{Y}$, we have $g_2^*(p,q)=\iota_{S}(p,q)$, with $S=\{(p,q)\in\mathcal{Y}:\|(p,q)\|_{\times *}\leq\lambda\}$. Moreover, $h$ is differentiable with $\kappa^{-1}:= 2\lambda$-Lipschitz continuous gradient. To solve this problem, we require the following formulae
\begin{align*}
	prox_{\gamma f}(x) &= \underset{y\in\mathbb{R}^n}{\arg\min}\big\{\gamma f(y) + \frac{1}{2}\|y-x\|^2\big\}=\underset{y\in [0,1]^n}{\arg\min}\big\{ \frac{1}{2}\|y-x\|^2\big\} = P_{[0,1]^n}(x) \forall x\in\mathbb{R}^n, 
\end{align*} 
\begin{align*}
	prox_{\gamma g_1^*}(p)&= \underset{y\in\mathbb{R}^n}{\arg\min}\big\{\gamma g_1^*(y) + \frac{1}{2}\|y-x\|^2\big\}= \underset{y\in\mathbb{R}^n}{\arg\min}\big\{\gamma (\iota_{[-1,1]^n}(y)+y^Tb) + \frac{1}{2}\|y-x\|^2\big\}\\
	&=\underset{y\in[-1,1]^n}{\arg\min}\big\{\gamma (y^Tb)+ \frac{1}{2}\|y-x\|^2\big\} = P_{[-1,1]^n} (p-\gamma b)  \; \forall p\in\mathbb{R}^n,
\end{align*}
\begin{align*}
	prox_{\gamma g_2^*}(p,q)&= P_S(p,q) \;\forall(p,q)\in\mathcal{Y},
\end{align*}
where $\gamma>0$ and the projection operator $P_S: \mathcal{Y}\rightarrow S$ is defined as (see [\cite{Bot2010}])
\begin{align*}
	(p_{i,j},q_{i,j})\mapsto \lambda\frac{(p_{i,j},q_{i,j})}{\max\big\{\lambda, \sqrt{p_{i,j}^2+q_{i,j}^2} \big\}}, 1\leq i \leq M, 1\leq j \leq N.
\end{align*}

Follows from the definition of the proximity operator of $f$ relative to the variable matrices (\ref{prox_variable_matrices}) and Lemma \ref{lem3.7(11)} $(iii)$, for $\gamma_n >0$, we obtain (see also (\ref{prox^U_f and P^U_C}))
\begin{align*}
	prox_{\gamma_n f}^{U_n^{-1}} (x) = J_{(U_n^{-1})^{-1}\partial \gamma_nf}(x)= J_{U_n\partial \gamma_nf}(x)=(U_n^{-1}+\partial\gamma_n f)^{-1}\circ U_n^{-1}
\end{align*}
and similarly for $i=1,2$
\begin{align*}
	prox_{\gamma_n  g_i^*}^{U_n^{-1}} (x) = J_{(U_n^{-1})^{-1}\partial \gamma_n g_i^*}(x)= J_{U_n\partial \gamma_n g_i^*}(x)=(U_n^{-1}+\partial\gamma_n g_i^*)^{-1}\circ U_n^{-1}.
\end{align*}

In Theorem \ref{Primal-Dual Tseng-EP with var matrices and error}, chose $(\tau_{n})_{n\in\mathbb{N}}$ and $(\sigma_i)_{1\leq i \leq m}$ in $(0,+\infty)$ such that  $U_n = \tau_{n} Id$ and $(\forall i\in\{1,\dots, m\})$ $U_{i,n} =\sigma_{i,n} Id$. Then (\ref{Primal-Dual 5_Algor}) reduce to the fixed metric methods (see related work in [\cite{Vu2013}]). Then the proximal operators turn into as follows
\begin{align*}
	prox_{\gamma_n f}^{(\tau_n Id)^{-1}} (x) &= \big[((\tau_n Id)^{-1}+\partial\gamma_n f)^{-1}\circ (\tau_n Id)^{-1} \big](x)=\big[ (\frac{1}{\tau_n} Id+\partial\gamma_n f)^{-1}\circ (\frac{1}{\tau_n} Id)\big] (x) \\
	&= \big[ (\frac{1}{\tau_n})^{-1} J_{\tau_n\partial\gamma_n f} \circ (\frac{1}{\tau_n} Id) \big] (x) \\ 
	&= \tau_n prox_{\tau_{n}\gamma_n f}\big(\frac{1}{\tau_n} x\big) \;\forall x\in\mathbb{R}^n , 
\end{align*}
similarly for $i=1,2$ we obtain
\begin{align*}
 	prox_{\gamma_n g^*_{1}}^{(\sigma_{1,n} Id)^{-1}} (p) &= \big[((\sigma_{1,n} Id)^{-1}+\partial\gamma_n g^*_{1})^{-1}\circ (\sigma_{1,n} Id)^{-1} \big](x)=\big[ (\frac{1}{\sigma_{1,n}} Id+\partial\gamma_n g^*_{1})^{-1}\circ (\frac{1}{\sigma_{1,n}} Id)\big] (x) \\
 	&= \big[ (\frac{1}{\sigma_{1,n}})^{-1} J_{\sigma_{1,n}\partial\gamma_n g^*_{1}} \circ (\frac{1}{\sigma_{1,n}} Id) \big] (x)\\ 
 	&= \sigma_{1,n} prox_{\sigma_{1,n}\gamma_n g^*_{1}}\big(\frac{1}{\sigma_{1,n}} x\big)=(\sigma_{1,n})  \; \forall p\in\mathbb{R}^n. 
\end{align*}
\begin{align*}
	prox_{\gamma_n g^*_{2}}^{(\sigma_{2,n} Id)^{-1}} (p,q) 
	&= \sigma_{2,n} prox_{\sigma_{2,n}\gamma_n g^*_{2}}\big(\frac{1}{\sigma_{2,n}} (p,q)\big)\\
	&=\sigma_{2,n}\lambda\frac{(\frac{p_{\bar{i},\bar{j}}}{\sigma_{2,n}},\frac{q_{\bar{i},\bar{j}}}{\sigma_{2,n}})}{\max\big\{\lambda, \sqrt{(\frac{p_{\bar{i},\bar{j}}}{\sigma_{2,n}})^2+(\frac{q_{\bar{i},\bar{j}}}{\sigma_{2,n}})^2} \big\}} \\
	&= \frac{(p_{\bar{i},\bar{j}},q_{\bar{i},\bar{j}})}{\max\big\{ 1, \frac{1}{\lambda}\sqrt{(\frac{p_{\bar{i},\bar{j}}}{\sigma_{2,n}})^2+(\frac{q_{\bar{i},\bar{j}}}{\sigma_{2,n}})^2} \big\}}  1\leq \bar{i} \leq M, 1\leq \bar{j} \leq N.
\end{align*}

When we want to measure the quality of the restored imaged, we use the tool known as \textit{ signal-to-noise ratio} (ISNR), which is given by (see [\cite{Chantas2008Galatsanos}])
\begin{align*}
\mbox{ISNR}_n = 10 \log_{10}\left(\frac{\|x-b\|^2}{\|x-x_n\|^2} \right),
\end{align*}

where $x$, $b$, and $x_n$ are the original, the observed noisy and the reconstructed image
at iteration $n \in \mathbb{N}$, respectively.

For the experiment, we considered the $256\times256$ cameraman image and constructed the blurred image by making use of a Gaussian blur operator of size $9\times 9$ and standard deviation 4. In order to obtain the blurred and noisy image, we added a zero-mean white Gaussian noise with standard deviation $10^{-3}$. Figure \ref{fig:method blur and deblur} shows the original cameraman image and the blurred and noisy one. It also shows the image reconstructed by the  algorithm after 1000 iterations, when taking as regularization parameter $\lambda = 0.003$, all error terms are zero, the variable metrics $U_n=\tau Id$, $U_{i,n}=\sigma_{i,n}Id$ and by choosing as parameters $\tau_{n}=1$, $\sigma_{1,n} = 0.1$, $\sigma_{2,n} =1$, a starting point $p_{1,-1} = p_{2,2,-1} = \bar{1}\times 0.4660$, $p_{2,1,-1} = (\bar{1},\bar{1})\times 0.4660$ where $\bar{1}= \begin{bmatrix}
	1 & 1 & \cdots & 1\\
	1 & 1 & \cdots & 1 \\
	\vdots & \vdots &\vdots & \vdots \\
	1 & 1 & \cdots & 1 
\end{bmatrix}_{256\times 256}$, $v_0 = 2\lambda$, $v_{1,0}=v_{2,0}=\bar{0}$ where $\bar{0}$ is a $256\times 256$ zero matrix  and $\gamma = \frac{1}{\sqrt{10}\mu(\beta+1)}$ where $\mu=1$, $\beta = 2\lambda+\sqrt{9}$ for $i\in\{1,2\}$.    

\begin{figure}[h!]
	\centering
	\begin{subfigure}[b]{0.3\linewidth}
		\caption{Original image}
		\includegraphics[width=\textwidth]{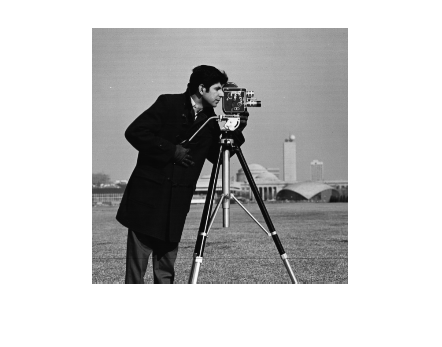}
	\end{subfigure}
	\begin{subfigure}[b]{0.3\linewidth}
		\caption{blurred and noisy image}
		\includegraphics[width=\linewidth]{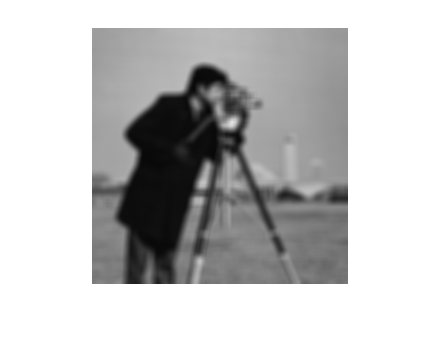}
	\end{subfigure}
	\begin{subfigure}[b]{0.3\linewidth}
	\caption{Reconstructed image}
	\includegraphics[width=\linewidth]{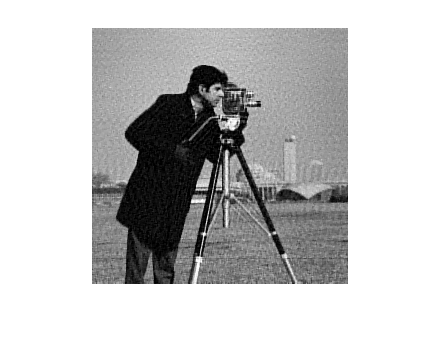}
	\end{subfigure}
	\caption{Figure (a) shows the original $256\times 256$ cameraman image, figure (b) shows the blurred and noisy image and figure (c) show the recover image generated by the algorithm after 1000 iterations.}
	\label{fig:method blur and deblur}
\end{figure}

In the error-free case such that the variable matrix is replaced by the identity matrix, we consider the the cameraman image with the same method of blurring with stopping criteria that is less than $10^{-2}$. For $n\geq 0$,  $\|x_{n}-x_{n+1}\|$, $|\mbox{fval}_{x_n}-\mbox{fval}_{x_{10000}^{**}}|$ and $\|x_n-x_{10000}^{**}\|$ are the examined criteria, where $\mbox{fval}_{x_n}$ is the objective value at the point $x_n$ and $x_{10000}^{**}$ is the solution point of the Tseng-EP algorithm after 10000 iterations. Table \ref{table:3 of criteria} shows the performance between the classical Tseng's algorithm and the Tseng's algorithm with extrapolation (Tseng-EP) when taking as regularization parameter $\lambda = 0.003$, a starting point $p_{1,-1} = p_{2,2,-1} = \bar{1}\times 0.4660$, $p_{2,1,-1} = (\bar{1},\bar{1})\times 0.4660$ where $\bar{1}= \begin{bmatrix}
	1 & 1 & \cdots & 1\\
	1 & 1 & \cdots & 1 \\
	\vdots & \vdots &\vdots & \vdots \\
	1 & 1 & \cdots & 1 
\end{bmatrix}_{256\times 256}$, $v_0 = 2\lambda$, $v_{1,0}=v_{2,0}=\bar{0}$ where $\bar{0}$ is a $256\times 256$ zero matrix  and  $\gamma_n = 1/(2\beta+0.1)$ where $\beta = 2\lambda+\sqrt{9}$.  We have seen that our proposed Tseng-EP algorithm spend the CPU-Time less than the classical Tseng's algorithm.

\begin{table}[h!]
	\centering
	\small
		\begin{tabular}{ c | c | c | c | c}
			 Criteria  & Algorithm& No.Iteration & ISNR & CPU-Time*  \\
	\hline \hline
		$\|x_n-x_{n+1}\| <10^{-2}$
			& Tseng
 			& 728
 			& 7.822241
 			& 6.55718

  \\ 
			  & Tseng-EP 
			  & 728 
			  & 7.822187
 			  & \cellcolor{LightGreen}4.53896
\\
		$| \mbox{fval}_{x_n}\!\!-\mbox{fval}_{x_{10000}^{**}} | < 10^{-2}$
			& Tseng
 			& 4802
 			& 7.423491
 			& 41.7736

  \\ 
			 & Tseng-EP 
			 & 4802
			 & 7.423493
 			 & \cellcolor{LightGreen}28.72458
\\
		$\|x_n-x_{10000}^{**}\| < 10^{-2}$
			& Tseng
 			& 4729
			& 7.424201
 			& 44.25134

  \\ 
			 & Tseng-EP 
			 & 4800 
			 & 7.424102
 			 & \cellcolor{LightGreen}29.23266
\\
			
		\end{tabular}
		\caption{The result of experiment for three different stopping criteria which are less than $10^{-2}$ when using $\gamma_n = 1/(2\beta+0.1)$ where $\beta = 2\lambda+\sqrt{9}$ and $\lambda = 0.003$ in error free case of the algorithm.}
		\label{table:3 of criteria}
	\end{table}

For the generalisation of our algorithm, we can choose $U_n = \tau_{n} Id$ and $(\forall i\in\{1,\dots, m\})$ $U_{i,n} =\sigma_{i,n} Id$, we select $\tau_n = 1$ and $\sigma_{i,n}$ is different values with the same setting of regularization parameter and initial points $p_{1,-1}$, $p_{2,2,-1}$, $p_{2,1,-1}$, $v_0$, $v_{1,0}$, $v_{2,0}$ and $\gamma_n = 1/(2\beta+0.1)$ where $\beta = 2\lambda+\sqrt{9}$ shown as Table \ref{table:change constant sigmas} for $n\geq 0$ for $i\in\{1,2\}$ which is the error free case of our Tseng-EP algorithm (see therein Theorem \ref{Primal-Dual Tseng-EP with var matrices and error}), we consider $|\mbox{fval}_{x_n}-\mbox{fval}_{x_{10000}^{**}}|<10^{-2}$ is a stopping criteria and notice that the choice of $\sigma_{i,n}$ for $i\in\{1,2\}$ should be a constant which is very close to 1, moreover if we replace them by 1.004 and 0.996, then we have seen that the algorithm not converges easily even more than 8000 iterations. Furthermore, Table \ref{table:change constant sigmas} give us the idea to use $\sigma_{i,n}$ are the convergent sequences which converges to 1 instead of the constant value. We used those sequences as follows:  $\frac{1}{k}$, $\frac{1}{k^2}$, $\frac{1}{k^5}$, $\frac{1}{k^k}$ and $\frac{k}{k+1}$ and demonstrate some result which consumed small of the number of iterations or give the best ISNR value. This illustrates in the Table \ref{table:shuffle of tau and sigma as seq}.

	\begin{table}[h!]
	\centering
	\small
	\begin{tabular}{  c | c | c | c | c | c | c}
		 Criteria  & Iteration & $\tau_n$ & $\sigma_{i,n}$ & $\mbox{fval}_{x_{iteration}}$ & ISNR & CPU-Time(s)  \\
		\hline \hline

		& 7232 
		& 1 
		& 1.003 
		&  97.7279 
		& 6.958979 
		& 46.0178 
		
		\\ 
		& 5589 
		& 1 
		& 1.002 
		& 97.727654 
		& 7.213882 
		& 35.0158 
		\\
		& 5000 
		& 1 
		& 1.001 
		& 97.727931 
		& 7.364356 
		& 31.1526 

		\\
		 $| \mbox{fval}_{x_n}\!\!-\mbox{fval}_{x_{10000}^{**}} | < 10^{-2}$ 
		& 4802 
		& 1 
		& 1 
		& 97.727766 
		& 7.423493 
		& 27.87 
		\\
		& 4992 
		& 1 
		& 0.9999 
		& 97.72766 
		& 7.367581 
		& 34.8569 
		\\ 
		& 5562 
		& 1 
		& 0.998 
		& 97.72766 
		& 7.222046 
		& 45.7339 
		\\
		& 7283 
		& 1 
		& 0.997 
		& 97.727572 
		& 6.956446 
		& 58.8615 

	\end{tabular}
	\caption{The result of experiment for the stopping criteria which are less than $| \mbox{fval}_{x_n}\!\!-\mbox{fval}_{x_{10000}^{**}} | < 10^{-2}$  when using $\gamma_n = 1/(2\beta+0.1)$ where $\beta = 2\lambda+\sqrt{9}$ and $\lambda = 0.003$ and diversify the constant value of $\sigma_{i,n}$ for the integer $ n\geq 0$ and $i\in\{1,2\}$.}
	\label{table:change constant sigmas}
\end{table}	

\begin{table}[h!]
	\centering
	\small
	\begin{tabular}{ c | c | c | c | c | c}
	 	Iteration & $\tau_n$ & $\sigma_{i,n}$ & $\mbox{fval}_{x_{iteration}}$ & ISNR & CPU-Time(s)  \\
		\hline \hline

		 4802 
		& 1 
		& 1 
		& 97.727766 
		& 7.423493 
		& 27.87 
		\\
		
		 4804 
		&  1
		& 1-($\frac{1}{k^2}$) 
		& 97.727863 
		& 7.423559 
		& 31.015 
		
		\\ 
		4803 
		& 1 
		& 1-($\frac{1}{k^5}$) 
		& 97.727786 
		& 7.423521 
		& 30.6754 
		\\
		 4802 
		& $\frac{k}{k+1}$
		& 1 
		& 97.727619 
		& \cellcolor{LightCyan} 7.423555 
		& 31.1954 

		\\
		  4803 
		& $\frac{k}{k+1}$ 
		& 1-($\frac{1}{k^5}$) 
		&  97.72767 
		&  7.423575 
		&  29.1343 
		\\
		 4802 
		&  1-($\frac{1}{k^2}$) 
		&  1 
		&  97.727403 
		&  7.42298 
		&  30.0258 
		\\ 
		 4802 
		&  1-($\frac{1}{k^5}$)
		&  1 
		&  97.727357 
		&  7.422935 
		&  30.1102 
		\\
		 4804 
		&  1-($\frac{1}{k^5}$) 
		&  1-($\frac{1}{k^2}$) 
		&  97.727494 
		&  7.42299 
		&  30.055 
		\\
		  4802 
		&  1-($\frac{1}{k}$) 
		&  1 
		&  97.727629 
		&  7.423016 
		&  30.9404 
		\\
		
		 4803 
		&  1-($\frac{1}{k}$) 
		&  1-($\frac{1}{k^5}$) 
		&  97.727648 
		&  7.422963 
		&  29.8768 
		\\
		
		 4803 
		&  1-($\frac{1}{k}$) 
		&  1-($\frac{1}{k^k}$) 
		&  97.72797 
		&  7.423079 
		&  30.7497 
	\end{tabular}
	\caption{The result of experiment when $\tau_n$, $\sigma_{i,n}$ are selected by the value between a constant 1 and the sequences which converges to 1 with stopping criteria  $| \mbox{fval}_{x_n}\!\!-\mbox{fval}_{x_{10000}^{**}} | < 10^{-2}$ by using $\gamma_n = 1/(2\beta+0.1)$  where $\beta = 2\lambda+\sqrt{9}$ and $\lambda = 0.003$ for all $n\geq0$, $i\in\{1,2\}$.}
	\label{table:shuffle of tau and sigma as seq}
\end{table}	

However, since $\sigma_{i,n}$ for $i\in\{1,2\}$ can be independent of choice, then we started experiment with fixing $\tau_n=1$, $\sigma_{1,n}=1$ with $\sigma_{2,n}$ are the sequence i.e., $\frac{k}{k+1}$, $(1+\frac{1}{k})^k$, 1-($\frac{1}{k}$), 1-($\frac{1}{k^2}$), 1-($\frac{1}{k^5}$). The experiment result are shown as in Table \ref{table:switching sigma by fixing 1}. Even though some results give us a little bit better of ISNR but they still consume the CPU-Time more than when we chose $\sigma_{i,n} = 1, \; \forall i\in\{1,2\}$.

\begin{table}[h!]
	\centering
	\small
	\begin{tabular}{  c | c | | c | c | c | c | c}
		 Iteration & $\tau_n$ & $\sigma_{1,n}$ & $\sigma_{2,n}$ & $\mbox{fval}_{x_{iteration}}$ & ISNR & CPU-Time(s)  \\
		\hline \hline

		 5431 
		& 1 
		& 1 
		& $\frac{k}{k+1}$ 
		& 97.727452 
		& 7.256485 
		& 33.3473 
		\\
		
		 4804 
		&  1
		&  1 
		&  1-($\frac{1}{k^2}$) 
		& 97.727866 
		& 7.423558 
		& 29.1967 
		\\
		 5432 
		&  1
		&  1 
		&  1-($\frac{1}{k}$) 
		& 97.727455 
		& 7.256488 
		& 33.4189 
		\\
		 4803 
		&  1
		&  1 
		&  1-($\frac{1}{k^5}$) 
		& 97.727787 
		& 7.42352 
		& 28.7536 
		\\
		 4802 
		&  1
		&  $\frac{k}{k+1}$ 
		&  1 
		& 97.72749 
		& 7.423239 
		& 30.2019 
		\\
		 4802 
		&  1
		&  1-($\frac{1}{k^2}$) 
		&  1 
		& 97.727763 
		& 7.423495 
		& 29.4369 
		\\
		 4802 
		&  1
		&  1-($\frac{1}{k}$) 
		&  1 
		& 97.727488 
		& 7.42324 
		& 28.8945 
		\\
		 4802 
		&  1
		&  1-($\frac{1}{k^5}$) 
		&  1 
		& 97.727765 
		& 7.423495 
		& 29.1847 
		\\
		 4802 
		&  1
		&  $\frac{k}{k+1}$ 
		&  1 
		& 97.72749 
		& 7.423239 
		& 30.5509 
		\\
		
	\end{tabular}
	\caption{the result of experiment when we fixed $\tau_n =1$ and shuffle $\sigma_{1,n}$ and $\sigma_{2,n}$ between 1, $\frac{k}{k+1}$, $1-(\frac{1}{k})$, $1-(\frac{1}{k^2})$, $1-(\frac{1}{k^5})$ and $(1+\frac{1}{k})^k$ with stopping criteria  $| \mbox{fval}_{x_n}\!\!-\mbox{fval}_{x_{10000}^{**}} | < 10^{-2}$ and $\gamma_n = 1/(2\beta+0.1)$ where $\beta = 2\lambda+\sqrt{9}$ and $\lambda = 0.003$ for all $n\geq0$, $i\in\{1,2\}$.}
	\label{table:switching sigma by fixing 1}
\end{table}


Again, we consider to solve this problem by the same setting of $U_n = \tau_{n} Id$ and $(\forall i\in\{1,\dots, m\})$ $U_{i,n} =\sigma_{i,n} Id$ for some selections of $ \tau_{n}$, $\sigma_{i,n}$ and regularization parameter $\lambda=0,003$ and initial points $p_{1,-1}$, $p_{2,2,-1}$, $p_{2,1,-1}$, $v_0$, $v_{1,0}$, $v_{2,0}$. But in this observation, the method is allowed to have errors. Indeed, $a_{1,n}$, $b_{1,n}$, $c_{1,n}$, $a_{2,i,n}$, $b_{2,i,n}$, $c_{2,i,n}$ are absolutely summable sequences. Then we need to select $\gamma_n$ which satisfied condition in Theorem \ref{Primal-Dual Tseng-EP with var matrices and error} $\big( (\gamma_n)_{n\in\mathbb{N}}\leq\lambda$ with $\lambda < \frac{1}{\sqrt{10}\mu\beta}$ and $\liminf\limits_{n\rightarrow +\infty} \gamma_n > 0 \big)$, so we choose   $\gamma_n = \frac{1}{\sqrt{10}\mu(\beta+1)}$ where $\beta = 2\lambda+\sqrt{9}$ which $\beta = 2\lambda+\sqrt{9}$. Table \ref{table:fixed tau and sigma but variant Error} shows the result when  all of error terms equal to the following sequences ${1}/{k^2}$, ${1}/{k^5}$, ${1}/{k^k}$, $(1/2)^k$ by fixed $\tau_n=\sigma_{i,n}=1$ for all $n\geq0$, $i\in\{1,2\}$. We observe that their performances are not much significantly different but it is obvious that they spend double time of the error-free case.  

\begin{table}[h!]
	\centering
	\small
	\begin{tabular}{  c | c | c | c | c | c | c}
		Iteration & $\tau_n$ & $\sigma_{i,n}$ & Error & $\mbox{fval}_{x_{iteration}}$ & ISNR & CPU-Time(s)  \\
		\hline \hline

		9976 
		& 1 
		& 1 
		& ${1}/{k^2}$ 
		& 97.727773 
		& 7.415526 
		& 65.3116 
		\\
		
		9972 
		&  1 
		&  1 
		&  ${1}/{k^5}$  
		& 97.727751 
		& 7.415609 
		& 64.946 
		\\
		9973 
		&  1 
		&  1 
		&  ${1}/{k^k}$ 
		& 97.727652 
		& 7.415548 
		& 65.146 
		\\
		9971 
		&  1 
		&  1 
		&  $(1/2)^k$  
		& 97.727843 
		& 7.415333 
		& 65.8735 

	\end{tabular}

	\caption{the result of experiment when we fixed $\tau_n =1$, $\sigma_{i,n}=1$  for $n\geq 0$, $i\in\{1,2\}$ and various errors with stopping criteria  $| \mbox{fval}_{x_n}\!\!-\mbox{fval}_{x_{10000}^{**}} | < 10^{-2}$ and $\gamma_n = 1/(\sqrt{10}\mu(\beta+1)$ where $\beta = 2\lambda+\sqrt{9}$ and $\lambda = 0.003$ for all $n\geq0$, $i\in\{1,2\}$.}
	\label{table:fixed tau and sigma but variant Error}
\end{table}	


Next let the number of iteration is fixed at 5,000 iterations and $\tau_{n}=1$, $\sigma_{i,n}=1$, then differ the error terms as ${1}/{k^2}$,  ${1}/{k^5}$,  ${1}/{k^k}$,  $({1}/{2})^k$ shown as in Table \ref{table:fixed iteration_5000 tau and sigma1 and vary error}. We can see again that the modification of error in our experiment does not have much effect to the result but when we look at ISNR they deliver more than 8 with the highest one is 8.344218. In contrast, the function value is slightly high compared with the previous results for $\gamma_n = 1/(\sqrt{10}\mu(\beta+1)$.   

\begin{table}[h!]
	\centering
	\small
	\begin{tabular}{  c | c | c | c | c | c | c}
		Iteration & $\tau_n$ & $\sigma_{i,n}$ & Error & $\mbox{fval}_{x_{iteration}}$ & ISNR & CPU-Time(s)  \\
		\hline \hline

		5000 
		& 1 
		& 1 
		& $\frac{1}{k^2}$ 
		& 99.465867 
		& \cellcolor{LightBlue} 8.344218 
		& 32.5193 
		\\
		
		5000 
		&  1 
		&  1 
		&  $\frac{1}{k^5}$ 
		& 99.459621 
		& 8.342828 
		& 32.4141 
		\\
		5000 
		&  1 
		&  1 
		&   $\frac{1}{k^k}$ 
		& 99.459901 
		& 8.34304 
		& 32.5323 
		\\
		5000 
		&  1 
		&  1 
		& $(\frac{1}{2})^k$ 
		& 99.460987 
		& 8.342093 
		& 32.5181 
		\\

	\end{tabular}
	\caption{the result of experiment after 5,000 iterations by fixing $\tau_n =1$, $\sigma_{i,n}=1$  and vary errors as  ${1}/{k^2}$,  ${1}/{k^5}$,  ${1}/{k^k}$,  $({1}/{2})^k$  and $\gamma_n = 1/(\sqrt{10}\mu(\beta+1)$ where $\beta = 2\lambda+\sqrt{9}$ and $\lambda = 0.003$ for all $n\geq0$, $i\in\{1,2\}$.}
	\label{table:fixed iteration_5000 tau and sigma1 and vary error}
\end{table}	

From the aforementioned trial, we plot the graph for 10000 iterations when we fixed $\tau_n =1$, $\sigma_{i,n}=1$ error terms is  ${1}/{k^2}$  and $\gamma_n = 1/(\sqrt{10}\mu(\beta+1)$ where $\beta = 2\lambda+\sqrt{9}$ and $\lambda = 0.003$ for all $n\geq0$, $i\in\{1,2\}$ shown as Figure \ref{fig_graph_peak}. We can detect the peak point by using \textit{findpeaks} in MATLAB to find the local maximum point and lastly we find that the maximum point is presented at 3736 iterations given the ISNR value equal to 8.467. However, we cannot confirm that this is the highest value of  ISNR because if we change our control parameters such as error terms, $\tau_{n}$ $\sigma_{i,n}\; \forall i\in\{1,2\}$ or even the stepsize $\gamma_{n}$, the highest ISNR value may be a different point.

\begin{figure}[h!]
	\centering
  \includegraphics[width=1\textwidth, height=12cm]{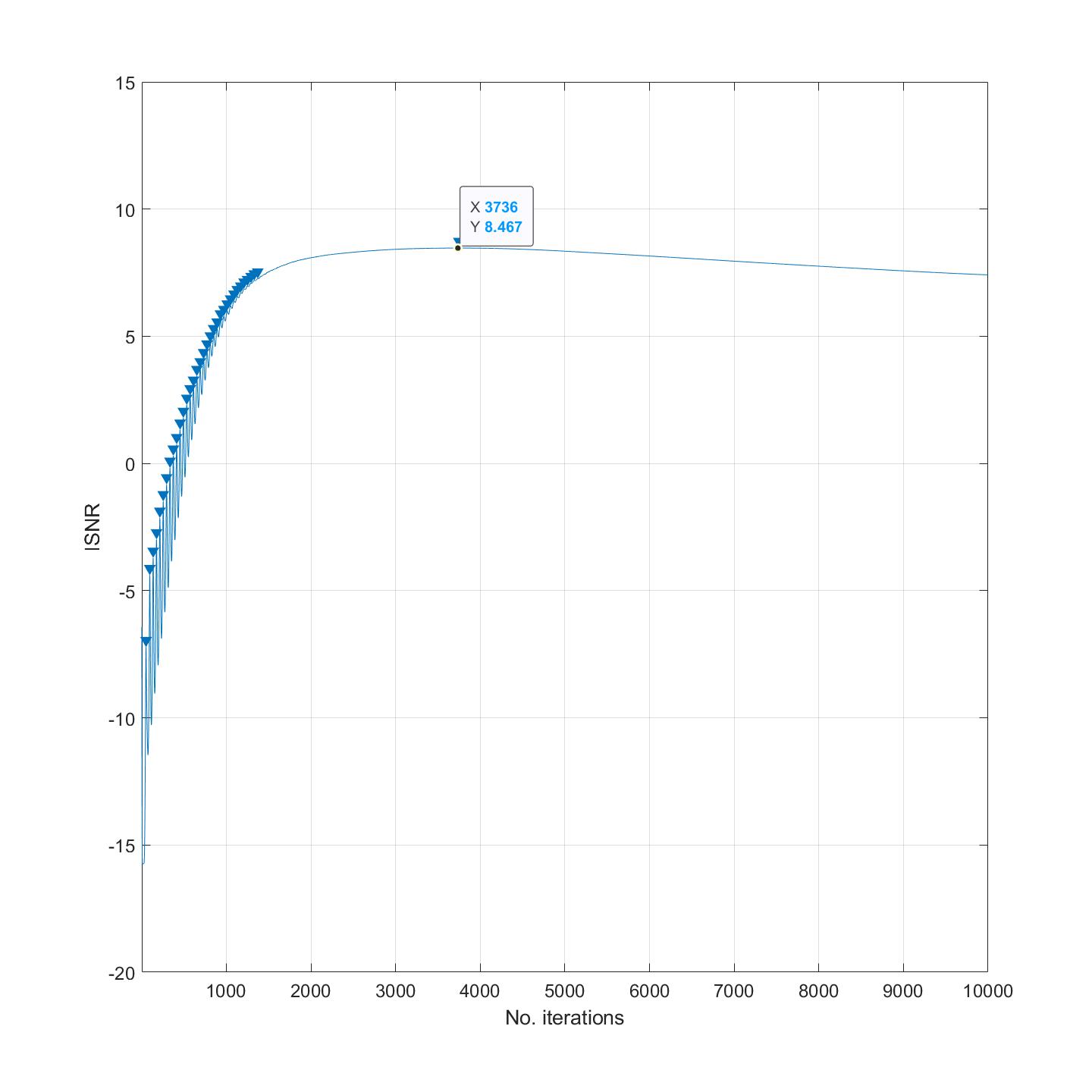}
	\vspace{-.85cm}
	\caption{The graph illustrates the ISNR value after 10000 iterations when we fixed $\tau_n =1$, $\sigma_{i,n}=1$ error terms is  ${1}/{k^2}$  and $\gamma_n = 1/(\sqrt{10}\mu(\beta+1)$ where $\beta = 2\lambda+\sqrt{9}$ and $\lambda = 0.003$ for all $n\geq0$, $i\in\{1,2\}$.}
	\label{fig_graph_peak}
\end{figure}

\noindent\textbf{Acknowledgment.} This work was supported by the Development and Promotion of Science and Technology Talents Project (DPST) scholarship, Thailand Government Scholarships. We are very thankful to Dr. habil. Ernö Robert Csetnek for his careful guidance.


\newpage
\nocite{*}

\begin{thebibliography}{99}

\bibitem{BC-Book}
Bauschke H. H.,  and Combettes P. L. (2011). Convex Analysis and Monotone Operator Theory in Hilbert Spaces. Springer, New York.

\bibitem{FISTA2009}
Beck A. and Teboulle M. (2009), A fast iterative shrinkage-tresholding algorithm for linear inverse problems.
SIAM J. Imaging Sci. 2(1), 183–202.

\bibitem{Aharon.et.al2001}
Ben-Tal A., Margalit T., and Nemirovski A. (2001). The ordered subsets mirror descent optimization method with applications to tomography. SIAM Journal on Optimization 12(1),  79-108.

\bibitem{Bohmetal2020}
Böhm A., Sedlmayer M., Csetnek E. R., and Boţ R. I. (2022). Two Steps at a Time---Taking GAN Training in Stride with Tseng's Method. SIAM Journal on Mathematics of Data Science, 4(2), 750-771.

\bibitem{Bot2015Csetnek}
Boţ R. I., and  Csetnek E. R. (2015). On the convergence rate of a forward-backward type primal-dual splitting algorithm for convex optimization problems. Optimization, 64(1), 5-23.

\bibitem{Bot2010}
Boţ R. I. (2009). Conjugate duality in convex optimization. Vol. 637, Springer Science \& Business Media.

\bibitem{AriasCombettes2011}
Briceño-Arias L. M., and Combettes P. L. (2011). A monotone + skew splitting model for composite monotone
inclusions in duality. SIAM J. Optim. 21(4), 1230–1250.

\bibitem{Burke1999Qian}
Burke J. V., and Qian M. (1999). A variable metric proximal point algorithm for monotone
operators. SIAM J. Control Optim. 37(2), 353–375.

\bibitem{Burke2000Qian}
Burke J. V., and Qian M. (2000). On the superlinear convergence of the variable metric
proximal point algorithm using Broyden and BFGS matrix secant updating. Math. Program.
88(1), 157–181.



		

\bibitem{Chambolle2004}
Chambolle A. (2004). An algorithm for total variation minimization and applications. Journal of Mathematical imaging and vision, 20(1), 89-97.		


\bibitem{Chantas2008Galatsanos}		
Chantas G., Galatsanos N., Likas A., and Saunders M. (2008). Variational bayesian image
restoration based on a product of $t$-distributions image prior. IEEE Trans. Image
Process. 17(10), 1795–1805.

\bibitem{Combettes2012Pesquet}
Combettes P. L., and Pesquet J. C. (2012). Primal-dual splitting algorithm for solving inclusions
with mixtures of composite, Lipschitzian, and parallel-sum type monotone operators. Set-Valued Var. Anal. 20:307–330.

\bibitem{Combettes2001}
Combettes P. L. (2001). Quasi-Fejérian analysis of some optimization algorithms. In Studies in Computational Mathematics (Vol. 8, pp. 115-152). Elsevier.

\bibitem{Combettes2014Vu}
Combettes P. L., and Vũ B. C. (2014). Variable metric forward-backward splitting with applications to monotone inclusions in duality. Optimization, 63(9), 1289-1318. DOI:10.1080/02331934.2012.733883

\bibitem{Variablemetic_quaisi_Combettes2013Vu}
Combettes P. L., and Vũ B. C. (2013). Variable metric quasi-Fejér monotonicity. Nonlinear Anal. 78, 17–31.


\bibitem{Gidel2018etal}
Gidel G., Berard H., Vignoud G., Vincent P., and Lacoste-Julien S. (2018). A variational inequality perspective on generative adversarial networks. arXiv preprint arXiv:1802.10551.


\bibitem{Hiriart-Urruty1993Lemarechal}
Hiriart-Urruty J. B., and  Lemar´echal C. (1993). Convex Analysis and Minimization Algorithms. Springer-Verlag,
New York, NY.

\bibitem{Kato1980}
Kato T. (1980), Perturbation Theory for Linear Operators, 2nd ed. Springer-Verlag, New York.


\bibitem{Malitsky2020Tam}
Malitsky Y., and Tam M.K. (2020). A forward-backward splitting method for monotone inclusions without cocoercivity. SIAM Journal on Optimization, 30(2), 1451-1472.


\bibitem{Parente2008Lotito}
Parente L. A., Lotito P. A., and Solodov M.,V. (2008). A class of inexact variable metric
proximal point algorithms. SIAM J. Optim. 19(1), 240–260.

\bibitem{Polyak1987}
Polyak  B. T. (1987), Introduction to Optimization, Optimization Software Inc., New York.

\bibitem{Popov}
Popov L. D. (1980). A modification of the Arrow-Hurwicz method for search of saddle points. Mathematical notes of the Academy of Sciences of the USSR, 28(5), 845-848.


\bibitem{Rockafellar}
Rockafellar R. (1967). Duality and stability in extremum problems involving convex functions. Pacific Journal of Mathematics 21(1), 167-187.

\bibitem{Rockafellar1976}
Rockafellar R. T. (1976). Monotone operators and the proximal point algorithm. SIAM J. Control Optim. 14(5), 877–898.


\bibitem{Shai et al.2012}
Shalev-Shwartz S. (2012). Online learning and online convex optimization. Foundations and Trends® in Machine Learning, 4(2), 107-194.


\bibitem{Tseng2000}
Tseng P. (2000). A modified forward-backward splitting method for maximal monotone
mappings. SIAM J. Control Optim. 38(2), 431–446.



\bibitem{Vu2013}
Vũ B. C. (2013). A splitting algorithm for dual monotone inclusions involving cocoercive operators. Advances in Computational Mathematics, 38(3), 667-681.

\bibitem{Vu2013Variable}
Vũ B. C., (2013). A variable metric extension of the forward–backward–forward algorithm for monotone operators. Numerical Functional Analysis and Optimization, 34(9), 1050-1065.



		

\end{thebibliography}

\end{document}